\PassOptionsToPackage{table}{xcolor}
\documentclass[journal, twocolumn]{IEEEtran}

\usepackage{amsmath,amsfonts}
\usepackage{array}
\usepackage{textcomp}
\usepackage{stfloats}
\usepackage{url}
\usepackage{graphicx}
\usepackage{subfiles}
\usepackage{subcaption}
\usepackage{booktabs}
\usepackage{multirow}
\usepackage{makecell}
\usepackage{algorithm}
\usepackage{algpseudocode}
\usepackage{enumitem}
\usepackage{tcolorbox}
\usepackage{cite}
\usepackage{amssymb}
\usepackage{xcolor}
\usepackage[hidelinks]{hyperref}
\usepackage{tabularx}
\usepackage{amsthm}
\usepackage{pifont}
\usepackage{balance}
\usepackage{tikz}
\usepackage{siunitx}
\usepackage{titlesec}
\titlespacing*{\section}{0pt}{5pt}{5pt} 
\titlespacing*{\subsection}{0pt}{5pt}{5pt} 

\extrafloats{256}

\theoremstyle{plain}
\newtheorem{theorem}{Theorem}[section]
\newtheorem{lemma}[theorem]{Lemma}
\newtheorem{proposition}[theorem]{Proposition}

\theoremstyle{definition}
\newtheorem{definition}[theorem]{Definition}

\theoremstyle{remark}
\newtheorem{remark}[theorem]{Remark}

\newcolumntype{Y}{>{\centering\arraybackslash}X}
\newcommand{\etal}{\textit{et al.}}

\makeatletter
\g@addto@macro\normalsize{%
  \setlength\abovedisplayskip{4pt plus 2pt minus 2pt}%
  \setlength\belowdisplayskip{4pt plus 2pt minus 2pt}%
  \setlength\abovedisplayshortskip{0pt plus 2pt}%
  \setlength\belowdisplayshortskip{4pt plus 2pt minus 2pt}%
}
\makeatother

\newcommand{\lex}{\mathrm{lex}}

\newcommand{\cl}{\mathrm{cl}}
\DeclareMathOperator*{\argmin}{arg\,min}

\DeclareMathOperator{\relint}{relint}

\usetikzlibrary{shapes.geometric, arrows.meta, positioning, fit, backgrounds, calc, decorations.pathreplacing, shapes.multipart}

\begin{document}

\title{Weight Certificates for Convex Multi-Objective MPC:\\ Geometric Characterization, $\ell^1$ Construction, and $\ell^2$ Foreclosure}

\author{Hadi Hajieghrary, Benedikt Walter, Chaitanya Shinde, Miguel Hurtado, and Jerry Lopez
\thanks{The authors are with TORC Robotics LLC, an independent subsidiary of Daimler Truck AG. Correspondence: {\tt\small Hadi.Hajieghrary@Torc.ai}.}%
\thanks{\emph{Research-prototype disclaimer:} This paper describes a research prototype evaluated exclusively in closed-loop computer simulation (\texttt{nuPlan-mini}). It does not describe or characterize any Torc Robotics production autonomous-driving system or any system deployed on public roads. The term ``safety'' labels optimization-priority tiers and does not constitute a collision-avoidance guarantee or a claim of compliance with any legal or regulatory standard. All findings are scoped to the simulation benchmark conditions described herein.}}

\markboth{IEEE Transactions on Intelligent Vehicles,~Vol.~XX, No.~X, Month~Year}%
{Hajieghrary \MakeLowercase{\textit{et al.}}: Weight Certificates for Convex Priority-Ordered MPC}

\maketitle

\begin{abstract}
Automated-driving rulebooks rank rule violations lexicographically, and model predictive control enforces that ranking either exactly, through $L{+}1$ sequential programs per tick, or approximately, through a weighted sum tuned by the separation heuristic $w_1\gg w_2\gg\cdots\gg w_L$. We show the heuristic answers the wrong question. For a convex priority-ordered program, a weighted sum reproduces the lexicographic optimum precisely when its weight, augmented by a unit performance coefficient, supports the upper image of the achievement map at the lexicographic point; the admissible weights form the unit-performance slice of an outward normal cone. Under hinge penalties this slice is a polyhedron obtained by projecting a scaled-KKT system, and a linear program returns an interior weight with a certified margin; under squared-hinge penalties no finite weight is exact whenever the limiting multiplier is nonzero, with violation along the local minimizer branch decaying as $O(1/w)$. Calibrated on held-out logs, the resulting weights have near-equal tier components in nine of eleven calibration-eligible scenario classes and roughly double legal-tier event precision against a matched heuristic weight in closed-loop nuPlan experiments on a 25-rule rulebook. The certificate is, however, pointwise: no single weight is valid across the sampled ticks of an episode, the median lifetime is one sampling interval (zero subsequent ticks at the native rate), and persistence tracks active-set stability. These findings motivate monitored weighted solves with selective cascade fallback, although the compliance-pattern monitor detects only a subset of measured lapses.
\end{abstract}


\section{Introduction}
\label{sec:introduction}
Trajectory planning for automated vehicles is governed by constraints partitioned into priority tiers. Within this planning hierarchy, safety rules, including collision avoidance and dynamic stability, are treated as the highest-priority objectives. Legal rules, such as speed limits and lane discipline, occupy the next tier, followed by comfort rules, including bounded acceleration and smooth steering.
The \emph{rulebook} formalism of Censi~\etal~\cite{censi2019} formalizes this prioritization: a planner minimizes violations of higher-priority rules before considering lower-priority ones, and resolves ties based on performance. This work investigates the model-predictive-control (MPC) implementation of this framework and characterizes the conditions under which a single optimization can replace an exact cascade.

We aggregate rule violations for a candidate trajectory $z$ into $L$ priority levels via functionals $V_i(z)=\sum_j [g_{i,j}(z)]_+^{p}$, with $p$ in $\{1,2\}$, and denote the performance objective by $J(z)$. The lexicographic cascade implements the priority semantics by minimizing $V_i$ at each stage over the optimal set, followed by $J$. This approach requires $L+1$ sequential optimizations per decision instance, which is computationally prohibitive at $10$ to $100$ Hz. Weighted-sum scalarisation instead reduces the hierarchy to a single optimization.
\begin{align}\label{eq:ws-intro}
 \min_{z}\; &J(z)+\sum_{i=1}^{L} w_i\,V_i(z),
 \qquad w\in\mathbb{R}^{L}_{++},
\end{align}
offering an idealized factor-$(L{+}1)$ reduction in the number of solves per tick.

Conventional implementations of (1) employ the large-ratio separation heuristic $w_1 \gg w_2 \gg \cdots \gg w_L$. The theory requires only strict positivity. Whether a given $w$ recovers the lexicographic optimum is a geometric membership question, decidable by a single linear feasibility program. Empirical evaluation on held-out driving logs shows that in nine of eleven calibration-eligible classes, the calibrated weights in Table 1 have near-equal tier components, with ratios $1:1:1$ and approximately $2:1:1$, in contrast to the deployed heuristic's $100:10:1$. Where the audit identifies regions that constrain the weights, the comfort-tier lower bound is approximately $89$, while the heuristic supplies $10$. The heuristic fails in the direction that breaks equivalence: priority is determined by membership, not magnitude.

We address the geometric structure of the set of weights for which weighted-sum scalarisation recovers the lexicographic optimum of a convex priority-ordered program, the use of this structure for deployable calibration, and the persistence of a calibrated weight under receding-horizon operation.

\emph{\textbf{1. Weight-certificate theory.}} Equivalent weights are characterized as a normal-cone slice: a weight reproduces the lexicographic optimum of a convex priority-ordered program \emph{iff} $(w,1)$ supports the upper image of the achievement map at $p^\star$ (Theorem~\ref{thm:main}). Under $\ell_1$ hinges the slice is polyhedral with an explicit scaled-KKT half-space construction in $\beta_{i,j}=w_i\alpha_{i,j}$ and a Chebyshev-centre weight with certified margin, verified against direct optimization on $1000$ randomized instances (Theorems~\ref{thm:poly} and~\ref{thm:l1comp}, Proposition~\ref{prop:convexJ}, Algorithm~\ref{alg:calib}); under $\ell_2$ squared hinges finite-weight exactness is foreclosed whenever the limiting multiplier is nonzero, with rates $\Theta(1/w_i)$ and $\Theta(1/w_i^{2})$ (Lemma~\ref{lem:vanish}, Theorem~\ref{thm:l2asymp}, Proposition~\ref{prop:l2tol}, Appendix~\ref{sec:appendix}).

\emph{\textbf{2. Certificate persistence.}} Multi-instance persistence is formalized through the robust intersection of pointwise regions (Proposition~\ref{prop:robust}) and measured on held-out receding-horizon problems: in the evaluated data, the intersections are empty in every scenario class, native-rate lifetime is typically shorter than one $0.1$\,s tick, and validity tracks active-set turnover (Section~\ref{sec:experiments:certificate}).

\emph{\textbf{3. Calibration and monitored deployment.}} Held-out geometry-calibrated weights (near-equal tier components in nine of eleven calibration-eligible classes) and a guarded weighted-solve architecture are evaluated against matched heuristic and cascade references, quantifying event-level fidelity, monitor accuracy, fallback demand, and computational cost on the eleven calibration-eligible classes (Section~\ref{sec:experiments:calibrated}); the reference cascade's priority semantics are validated on a fixed nuPlan benchmark ($12$ of $16$ scenarios, $p=0.0384$), with the closed-loop $\ell_2$ ablation moving in the direction the dichotomy predicts (Section~\ref{sec:experiments}).

The closest related deployment, W-SQP~\cite{hajieghrary2026wsqp}, originates from the same research group; the relationship is disclosed explicitly, and artifact-level overlap is enumerated in a statement released with the reproducibility artifact. W-SQP compiles nine driving-rule families into a four-tier shared-slack NMPC with heuristically separated quadratic tier penalties and names an $\ell_1$ lexicographic guarantee as an open refinement. The present paper supplies that refinement: it poses the equivalence question W-SQP left open, answers it with the support-normal characterization and the $\ell_1$ half-space construction, and shows that W-SQP's quadratic-penalty regime is the $\ell_2$ setting whose finite-weight exactness the dichotomy theorem rules out at binding constraints.

\noindent\emph{Note on terminology:} Throughout this paper, the terms ``safety,'' ``safety rule,'' and ``safety level'' refer to constraint tiers that receive the highest optimization priority within the LCP-MPC framework. This terminology reflects the structure of the planning objective and does not constitute a representation that the research prototype is collision-free, that it satisfies any legal or regulatory safety standard, or that any deployment relying on this approach is safe for public-road operation. All empirical results are derived from closed-loop computer simulation and are scoped accordingly.

\subsection{Related Work}
\label{sec:related_work}

In convex scalarisation, a strictly positive weight is an inner normal of a supporting hyperplane at the exposed image point, and the supporting functionals there form the outward normal cone~\cite{boyd2004convex}. The operations-research strand poses our question directly: goal programming; Sherali's existence result, full characterization of the weight set, and two constructions, for lexicographic \emph{linear} programs~\cite{sherali1982equivalent}; explicit weight bounds for combinatorial instances~\cite{eisenbrand2003}; counterexamples admitting no finite weighting~\cite{sheralisoyster1983}; and a first-order characterization through stationarity~\cite{rentmeesters1996}. We differ on three counts: the setting is \emph{convex}, not linear; the admissible set is exhibited \emph{geometrically}, as the unit-performance slice of the upper image's outward normal cone, which makes its dimension, boundaries, and inscribed-ball margin computable objects rather than by-products of a construction; and where~\cite{rentmeesters1996} gives first-order conditions for the lexicographic problem, our scaled-KKT projection characterizes the \emph{weight set} reproducing its solution. The $\ell_2$ foreclosure has no ancestor here --- a Sherali--Soyster-type impossibility tied to the penalty, not to a pathological instance --- while Sherali's existence is recovered as non-emptiness of the slice.

For exact hierarchy at solver cost, the literature suggests embedding of the priority order in the receding-horizon loop through cascades of single-objective programs~\cite{longgatzke2007, ocampo2008,padhiyar2009,zheng2010lexnmpc,zavala2012utopia}, online relaxation programs~\cite{vada2001jota,vada2001auto}, or hierarchical QP cascades from robotics~\cite{kanoun2011,escande2014,mifsud2021}. All are exact, all pay $L{+}1$ solves or a relaxation, and none state when one solve suffices. The mechanism is exact-penalty theory: a hinge is exact at \emph{finite} weight above the dual norm of the optimal multiplier~\cite{nocedal2006,bertsekas2016}, whereas a squared hinge has zero boundary derivative and, at a nonzero limiting multiplier, is exact at no finite weight. This threshold is classical for one soft constraint in MPC~\cite{kerrigan2000soft}; lifting it to $L$ levels replaces the scalar threshold by a full-dimensional polyhedral \emph{region}. That region is not a critical region of explicit mpMPC~\cite{bemporad2002explicit}: both partition by active set, but explicit MPC partitions the \emph{state} space at a fixed objective, ours the \emph{weight} space at a fixed decision instance.

Rulebooks~\cite{censi2019} descend from minimum-violation planning~\cite{tumova2013,esterle2020}, and since gradient methods cannot optimize a lexicographic order directly, the field split along the smooth/non-smooth axis formalized here: combinatorial ranking~\cite{veer2023}, barrier encodings~\cite{xiao2021}, and continuous lexicographic surrogates~\cite{rasekhipour2018} --- the last being exactly the regime our $\ell_2$ analysis forecloses. Responsibility-Sensitive Safety~\cite{shalev2017rss,hasuo2022rss,ieee2846} is complementary (Section~\ref{sec:discussion:rss}).

\section{Problem Formulation}
\label{sec:problem_formulation}
We consider the exact equivalence between lexicographic and weighted-sum formulations in convex priority-ordered programs. Specifically, we characterize the scalar weights for which a single weighted-sum solve recovers the lexicographic constraint order. We define the priority-ordered control problem, its lexicographic cascade, and the weighted-sum surrogate, and introduce the upper-image geometry.
All spaces are finite-dimensional and Euclidean.

\subsection{Notation and Problem Formulation}
\label{sec:problem:notation}

The non-negative orthant of $\mathbb{R}^L$ is $\mathbb{R}_{\ge0}^L$ and its interior is $\mathbb{R}_{++}^L$. The positive part is $[a]_+:=\max(0,a)$ and the relative interior of a set $S$ is $\relint(S)$. For a closed convex set $\mathcal{S}\subset\mathbb{R}^N$ and a point $p\in\mathcal{S}$, the \emph{normal cone} is
\begin{equation}\label{eq:normalcone}
  N_{\mathcal{S}}(p):=\bigl\{\eta\in\mathbb{R}^N:\langle\eta,q-p\rangle\le0,\;\forall q\in\mathcal{S}\bigr\},
\end{equation}
and the \emph{outward normal cone} is $N_{\mathcal{S}}^+(p):=-N_{\mathcal{S}}(p)$.

We consider the discrete-time system $x_{k+1}=f(x_k,u_k)$ over horizon $N$, with $x_k\in\mathbb{R}^n$ and $u_k\in\mathbb{R}^m$, and collect the trajectory $z:=(x_{0:N},u_{0:N-1})\in\mathbb{R}^d$ where $d:=n(N{+}1)+mN$. The performance objective is the Bolza cost
\begin{equation}\label{eq:perfobj}
  J(z):=\sum_{k=0}^{N-1}\ell(x_k,u_k)+\ell_f(x_N),
\end{equation}
with stage cost $\ell$ and terminal cost $\ell_f$. The \emph{dynamics-feasible set} is
\[
  \mathcal{Z}:=\bigl\{z\in\mathbb{R}^d:x_{k+1}=f(x_k,u_k),\;x_0=\bar{x}_0,\;u_k\in\mathcal{U}_{\mathrm{hard}}\bigr\},
\]
where $\mathcal{U}_{\mathrm{hard}}$ encodes physical actuator limits. The rule-based constraints are partitioned into $L$ priority levels
\begin{equation}\label{eq:levels}
  \mathcal{C}_i:=\bigl\{z:g_{i,j}(z)\le0,\;j=1,\ldots,m_i\bigr\},\quad i=1,\ldots,L,
\end{equation}
with priority ordering $\mathcal{C}_1\succ\mathcal{C}_2\succ\cdots\succ\mathcal{C}_L$. The
\emph{violation functional} at level~$i$ is
\begin{equation}\label{eq:viol}
  V_i(z):=\sum_{j=1}^{m_i}\bigl[g_{i,j}(z)\bigr]_+^p,\qquad p\in\{1,2\},
\end{equation}
The parameter $p$ selects the penalty form. The case $p=1$ is the $\ell_1$ hinge penalty; $p=2$ is the $\ell_2$ squared-hinge penalty. This distinction is structural and enters all subsequent analysis.

\subsection{Cascade, Weighted Sum, and Standing Assumptions}
\label{sec:problem:cascade}

The violation vector induces the lexicographic preference relation
$(V_1(z),\ldots,V_L(z))$.

\begin{definition}[Lexicographic preference]\label{def:lexpref}
  Trajectory $z$ is lexicographically preferred to $z'$, written $z\prec_{\lex}z'$, if
  there exists $k\in\{1,\ldots,L\}$ such that $V_i(z)=V_i(z')$ for all $i<k$ and
  $V_k(z)<V_k(z')$.
\end{definition}

The \emph{lex cascade} is the sequential realization of this preference. We set $\mathcal{F}_0:=\mathcal{Z}$, and for $i=1,\ldots,L$, \begin{align}\label{eq:cascade}
  V_i^\star&:=\min_{z\in\mathcal{F}_{i-1}}V_i(z),\\
  \mathcal{F}_i&:=\bigl\{z\in\mathcal{F}_{i-1}:V_i(z)=V_i^\star\bigr\}.\nonumber
\end{align}
The lex optimum is the performance minimizer on the final slice,
\begin{equation}\label{eq:lexopt}
z_{\lex}^\star\in\argmin_{z\in\mathcal{F}_L}J(z).
\end{equation}
For a weight vector $w\in\mathbb{R}_{++}^L$, the \emph{weighted-sum} (WS) problem is
\begin{equation}\label{eq:ws}
z_{\mathrm{ws}}^\star(w)\in\argmin_{z\in\mathcal{Z}}\;J(z)+\sum_{i=1}^{L}w_i V_i(z).
\end{equation}
The cascade requires $L+1$ sequential programs per MPC cycle. We characterize the conditions under which the single weighted-sum solve~\eqref{eq:ws} yields the same solution as the cascade, that is, when
$z_{\mathrm{ws}}^\star(w)=z_{\lex}^\star$.

We adopt the following assumptions throughout.

\noindent\textbf{(A1) Existence.} \textit{The set $\mathcal{Z}$ is compact; the minima in\,\eqref{eq:cascade} are attained, the argmin in\,\eqref{eq:lexopt} is non-empty, and the minimiser in\,\eqref{eq:ws} exists for every $w\in\mathbb{R}_{++}^L$.}

\noindent\textbf{(A2) Uniqueness.} \textit{The lex optimum $z_{\lex}^\star$ is unique, and the image point $p^\star:=(V_1^\star,\ldots,V_L^\star,J^\star)$ is uniquely determined.}

\noindent\textbf{(A3) Convexity.} \textit{The set $\mathcal{Z}$ is convex; each constraint function $g_{i,j}$ is convex; the performance objective $J$ is convex.}

Under~\textup{(A3)}, each $V_i$ is convex because $[a]_+^p$ is convex and
Non-decreasing convexity for $p$ in $\{1,2\}$ ensures that $V_i$ and $J$ are continuous, as is the achievement map $\Phi$. We impose compactness of $\mathcal{Z}$ as a uniform sufficient condition for existence. Coercivity of $J+\sum_i V_i$ is not sufficient to guarantee existence for all weighted sums $J+\sum_i w_i V_i$ with $w$ in $\mathbb{R}_{++}^L$, since for sufficiently small positive weights, descent directions of $J$ can remain unpenalized even when the equal-weight sum is bounded below.
Convexity of $\mathcal{Z}$ is the restrictive requirement. It holds for affine dynamics, or after reformulation to an equivalent convex feasible trajectory set. We treat (A3) as the exact scope of the theory. Deployment leaves this scope globally and re-enters it locally by linearizing the kinematic-bicycle dynamics at each inner solve.

\subsection{The Upper Image and the Lexicographic Point}
\label{sec:problem:upper}

The central geometric object is the image of $\mathcal{Z}$ under the \emph{achievement map}
\begin{equation}\label{eq:achieve}
  \Phi(z):=\bigl(V_1(z),V_2(z),\ldots,V_L(z),J(z)\bigr)\in\mathbb{R}^{L+1}.
\end{equation}

\begin{definition}[Upper image]\label{def:upper}
  The upper image of the achievement map is
  \[
    \overline{\mathcal{P}}:=\cl\bigl\{(y,t)\in\mathbb{R}^{L+1}:\exists z\in\mathcal{Z},\;V_i(z)\le y_i\;\forall i,\;J(z)\le t\bigr\},
  \]
  where $\cl(\cdot)$ denotes topological closure in $\mathbb{R}^{L+1}$.
\end{definition}

The upper image is standard in vector minimisation~\cite{ehrgott2005}:
it contains every achievable image point and every componentwise-worse point.

\begin{lemma}[Convexity of the upper image]\label{lem:convex}
  Under~\textup{(A3)}, the upper image $\overline{\mathcal{P}}$ is convex.
\end{lemma}

\begin{proof}
  Convexity is maintained under closure, and the pre-closure set
  $\{(y,t):\exists z\in\mathcal{Z},\,V_i(z)\le y_i,\,J(z)\le t\}$ is convex
  because the convex combination of two witnesses, which lies in
  $\mathcal{Z}$ by~(A3), witnesses the combination by convexity of each
  $V_i$ and of $J$.
\end{proof}

The lex cascade traces a sequence of nested upper images induced by lex-optimal trajectory slices. For each $i$ we define
\begin{equation}\label{eq:nested}
  \overline{\mathcal{P}}_i:=\cl\bigl\{(y,t):\exists z\in\mathcal{F}_i,\;V_{i'}(z)\le y_{i'}\;\forall i',\;J(z)\le t\bigr\},
\end{equation}
so that $\overline{\mathcal{P}}=\overline{\mathcal{P}}_0\supseteq\overline{\mathcal{P}}_1\supseteq\cdots\supseteq\overline{\mathcal{P}}_L$, and the lex optimum has image
\begin{equation}\label{eq:leximage}
  p^\star:=\Phi(z_{\lex}^\star)=(V_1^\star,\ldots,V_L^\star,J^\star)\in\overline{\mathcal{P}}_L.
\end{equation}

\begin{remark}
  We avoid calling the $\overline{\mathcal{P}}_i$ ``faces'' of $\overline{\mathcal{P}}$. An exposed face of a convex set is defined by linear support; the set   $\overline{\mathcal{P}}_i$ in~\eqref{eq:nested} does not have this form, because it admits upward slack in coordinates already minimized at earlier cascade stages. We therefore use   ``nested upper image'' for the lex-restricted construction and reserve ``face'' for the
  linear-support sense.
\end{remark}

The lexicographic image is an extreme point. Polyhedrality, established below, yields a full-dimensional outward normal cone under $\ell_1$.

\begin{lemma}[Lex point is an extreme point]\label{lem:extreme}
  Under~\textup{(A1)--(A3)}, the image point $p^\star$ is an extreme point of
  $\overline{\mathcal{P}}$.
\end{lemma}

\begin{proof}
  Suppose for contradiction $p^\star=\lambda p^{(1)}+(1-\lambda)p^{(2)}$ for some   $\lambda\in(0,1)$ and $p^{(1)},p^{(2)}\in\overline{\mathcal{P}}$ with   $p^{(1)}\ne p^{(2)}$. By Definition~\ref{def:upper}, each $p^{(j)}$ is the limit of a sequence $p^{(j)}_n=\Phi(z^{(j)}_n)+r^{(j)}_n$ with $z^{(j)}_n\in\mathcal{Z}$ and $r^{(j)}_n\in\mathbb{R}_{\ge0}^{L+1}$. By compactness of $\mathcal{Z}$~(A1) we may pass to a subsequence along which $z^{(j)}_n\to z^{(j)}_\infty\in\mathcal{Z}$, and by continuity of $\Phi$ the slacks converge to $r^{(j)}_\infty:=p^{(j)}-\Phi(z^{(j)}_\infty)\in\mathbb{R}_{\ge0}^{L+1}$; every upper-image
  point thus admits a witness in $\mathcal{Z}$. The first coordinate of  $\overline{\mathcal{P}}$ is bounded below by $V_1^\star$, so $p^{(j)}_1\ge V_1^\star$ for $j=1,2$; their convex combination equals $V_1^\star$, forcing both to equal $V_1^\star$ and, through $p^{(j)}_1=V_1(z^{(j)}_\infty)+r^{(j)}_{\infty,1}$ with both terms bounded below, $V_1(z^{(j)}_\infty)=V_1^\star$ and $r^{(j)}_{\infty,1}=0$, i.e.  $z^{(j)}_\infty\in\mathcal{F}_1$. Iterating coordinate by coordinate along $2,\ldots,L$ gives $z^{(j)}_\infty\in\mathcal{F}_L$; at the performance coordinate, (A2) gives a unique minimum $J^\star$ on $\mathcal{F}_L$, forcing $J(z^{(j)}_\infty)=J^\star$. Hence $p^{(j)}=p^\star$ for both $j$, contradicting $p^{(1)}\ne p^{(2)}$. \end{proof}

Three named settings organize the results. In \emph{Setting~(P)} (polyhedral), $\mathcal Z$ is polyhedral, the $g_{i,j}$ are affine, $J$ is convex and piecewise affine, and $V_i=\sum_j[g_{i,j}]_+$; the upper image and its normal cone then inherit polyhedral structure. \emph{Setting~(P$'$)} relaxes $J$ to a convex continuously differentiable function (its full hypotheses are stated with Proposition~\ref{prop:convexJ}); the deployed inner problems of Section~\ref{sec:deployment} live in (P$'$). In \emph{Setting~(Q)}, $V_i=\sum_j[g_{i,j}]_+^2$: each $V_i$ is continuously differentiable, and a squared hinge supplies no first-order force at its boundary.

\section{The Equivalence Region and Its Dichotomy}
\label{sec:equivalence_region}
This section states the paper's central results in one dependency-ordered
arc, from the equivalence cone and the support-normal characterization
through the $\ell_1$ construction and the $\ell_2$ foreclosure to the
robust multi-instance region.

\subsection{The Equivalence Cone and the Normalized Region}
\label{sec:method:cone}

Weighted-sum scalarisation with weight $w$ selects the boundary point of the upper image exposed by the linear functional $\langle(w,1),\cdot\rangle$. The lexicographic point $p^\star$ is exposed by this functional---and hence recovered by the weighted sum---exactly when $(w,1)$ supports $\overline{\mathcal{P}}$ at $p^\star$, i.e. when $(w,1)$ lies in the outward normal cone.

\begin{definition}[Equivalence cone]\label{def:cone}
  The equivalence cone at $p^\star$ is the outward normal cone of the upper image at the lex point, restricted to non-negative rule weights and a positive performance coefficient,
  \begin{equation}\label{eq:cone}
    \widehat{\Omega}(p^\star):=
    N_{\overline{\mathcal{P}}}^+(p^\star)
    \cap\bigl(\mathbb{R}_{\ge0}^L\times\mathbb{R}_{>0}\bigr)
    \subset\mathbb{R}^{L+1}.
  \end{equation}
\end{definition}

\begin{definition}[Normalised equivalence region]\label{def:region}
  The normalized equivalence region at $p^\star$ is the unit-performance slice
  \begin{equation}\label{eq:region}
    \Omega(p^\star):=
    \bigl\{w\in\mathbb{R}_{++}^L:(w,1)\in
    N_{\overline{\mathcal{P}}}^+(p^\star)\bigr\}.
  \end{equation}
\end{definition}

A key distinction is that $\Omega(p^\star)$ is not a cone in $\mathbb{R}^L$. It is the intersection of a cone in $\mathbb{R}^{L+1}$ with the affine hyperplane $\{(w,1):w\in\mathbb{R}^L\}$, and therefore is not closed under positive scaling of $w$. The genuinely conical object is $\widehat{\Omega}(p^\star)$.

\subsection{Support-Normal Equivalence}
\label{sec:theory:main}

\begin{theorem}[Support-normal equivalence]\label{thm:main}
Under \textup{(A1)--(A3)}, for every $w\in\mathbb{R}_{++}^L$,
\[
z_{\lex}^\star\in
\argmin_{z\in\mathcal{Z}}
\Bigl[J(z)+\textstyle\sum_{i=1}^{L}w_iV_i(z)\Bigr]
\iff
w\in\Omega(p^\star).
\]
If, in addition, the WS minimizer is unique---for instance, when
$J+\sum_iw_iV_i$ is strictly convex on $\mathcal Z$---then
$z_{\mathrm{ws}}^\star(w)=z_{\lex}^\star$ for every
$w\in\Omega(p^\star)$.
\end{theorem}

\begin{proof}
\emph{Forward.}
If $w\in\Omega(p^\star)$, then
$(w,1)\in N_{\overline{\mathcal{P}}}^+(p^\star)$; equivalently,
$\langle(w,1),q-p^\star\rangle\ge0$ for every
$q\in\overline{\mathcal P}$. For any $z\in\mathcal Z$,
$\Phi(z)\in\overline{\mathcal P}$, and hence
\[
 J(z)+\sum_iw_iV_i(z)
 \ge
 J(z_{\lex}^\star)+\sum_iw_iV_i(z_{\lex}^\star).
\]
Thus $z_{\lex}^\star$ is a WS minimizer.

\emph{Reverse.}
If $z_{\lex}^\star$ is a WS minimizer, then
$\langle(w,1),\Phi(z)-p^\star\rangle\ge0$ for every
$z\in\mathcal Z$. Any $q\in\overline{\mathcal P}$ is a limit of
$q_n=\Phi(z_n)+r_n$ with $z_n\in\mathcal Z$ and
$r_n\in\mathbb R_{\ge0}^{L+1}$. Since $(w,1)>0$,
\[
 \langle(w,1),q_n-p^\star\rangle
 =
 \langle(w,1),\Phi(z_n)-p^\star\rangle
 +\langle(w,1),r_n\rangle
 \ge0.
\]
Passing to the limit gives
$(w,1)\in N_{\overline{\mathcal P}}^+(p^\star)$, and therefore
$w\in\Omega(p^\star)$.
\end{proof}

Finite equivalence therefore corresponds to support at $p^\star$, and
$\Omega(p^\star)$ is the unit-performance slice of the supporting-normal
cone. The existence result of Sherali~\cite{sherali1982equivalent} and the
bounds of Eisenbrand~\etal~\cite{eisenbrand2003} can be read as statements
about the non-emptiness of, or sufficient membership conditions for, this
slice.

\subsection{\texorpdfstring{The $\ell_1$ Regime: a Computable Polyhedron}{The L1 Regime: a Computable Polyhedron}}
\label{sec:theory:regimes}

Setting~(P) of Section~\ref{sec:problem_formulation} gives the region
polyhedral structure.

\begin{theorem}[Polyhedrality of the region]\label{thm:poly}
Under Setting~\textup{(P)} and \textup{(A1)--(A3)}, the upper image
$\overline{\mathcal P}$ is a full-dimensional polyhedron; $p^\star$ is a
vertex, and hence $N_{\overline{\mathcal P}}^+(p^\star)$ is
full-dimensional; and $\Omega(p^\star)$ is a polyhedral region in
$\mathbb R^L$. It is non-empty if and only if
\[
 \{(w,1):w\in\mathbb R_{++}^L\}
 \cap N_{\overline{\mathcal P}}^+(p^\star)\ne\varnothing,
\]
and it has non-empty interior in $\mathbb R^L$ if and only if the same
unit-performance hyperplane meets
$\operatorname{int}N_{\overline{\mathcal P}}^+(p^\star)$ inside
$\mathbb R_{++}^L\times\{1\}$.
\end{theorem}

\begin{proof}
Represent each hinge by its epigraph: introducing $s_{i,j}\ge0$ with
$s_{i,j}\ge g_{i,j}(z)$ writes each level violation as the linear form
$\sum_js_{i,j}$ over a lifted feasible set cut out by affine constraints,
including the epigraph of the piecewise-affine $J$. The pre-closure set in
Definition~\ref{def:upper} is therefore a linear projection of a
polyhedron, augmented by $\mathbb R_{\ge0}^{L+1}$. Linear projections of
polyhedra are polyhedra, and polyhedra are closed, so
$\overline{\mathcal P}$ is polyhedral. It is full-dimensional because it
contains $p^\star+\mathbb R_{\ge0}^{L+1}$. By
Lemma~\ref{lem:extreme}, $p^\star$ is an extreme point, hence a vertex.
The normal cone at a vertex of a full-dimensional polyhedron is
full-dimensional. Finally,
$\Omega(p^\star)$ is the intersection of that polyhedral cone with the
unit-performance hyperplane, read in the $w$ coordinates. The stated
non-emptiness and full-dimensionality conditions follow directly.
\end{proof}

\begin{remark}[Augmented LICQ]\label{rem:auglicq}
LICQ is not required for Theorem~\ref{thm:poly}. When the boundary-hinge
gradients together with the active hard-inequality and equality gradients
are linearly independent (\emph{augmented LICQ}), the scaled subgradient and
hard-constraint multipliers are uniquely determined affinely by $w$; without
it, the region remains a polyhedral projection but may be lower-dimensional.
\end{remark}

In the polyhedral case (with $\ell_1$ penalties, affine constraints, and convex piecewise-affine $J$), the region admits an explicit half-space description. Let $\{a_1,\ldots,a_K\}$ denote the outward normals to the facets of $\overline{\mathcal{P}}$ incident to $p^\star$, so that
$N_{\overline{\mathcal{P}}}^+(p^\star)=\{\sum_k\mu_k a_k:\mu_k\ge0\}$. The region is the slice
\begin{equation}\label{eq:halfspace1}
  \Omega(p^\star)=
  \Bigl\{w\in\mathbb{R}_{++}^L:
  \exists\,\mu\ge0,\ \textstyle\sum_k\mu_k a_k=(w,1)\Bigr\},
\end{equation}
and eliminating $\mu$ yields
\begin{equation}\label{eq:halfspace2}
  \Omega(p^\star)=
  \bigl\{w\in\mathbb{R}_{++}^L:Cw\le d,\ Ew=f\bigr\}.
\end{equation}

For the KKT construction used with a differentiable convex $J$, define the strictly violated and boundary-binding hinge sets at the lexicographic solution by
\[
 \mathcal I_i^+:=\{j:g_{i,j}(z_{\lex}^\star)>0\},
 \qquad
 \mathcal I_i^0:=\{j:g_{i,j}(z_{\lex}^\star)=0\}.
\]
At a boundary-binding hinge, write
$\alpha_{i,j}\nabla g_{i,j}(z_{\lex}^\star)$ with
$\alpha_{i,j}\in[0,1]$, and introduce the \emph{scaled subgradient variable}
\begin{equation}\label{eq:scaled-subgradient}
  \beta_{i,j}:=w_i\alpha_{i,j},
  \qquad 0\le\beta_{i,j}\le w_i.
\end{equation}
Writing the active hard inequalities and equalities locally as
$h_a(z)\le0$ and $q_e(z)=0$, respectively, the stationarity condition becomes
\begin{align}
\label{eq:scaled-kkt}
0={}&\nabla J(z_{\lex}^\star)\\
&+\sum_{i=1}^{L}\left(
 w_i\sum_{j\in\mathcal I_i^+}\nabla g_{i,j}(z_{\lex}^\star)
 +\sum_{j\in\mathcal I_i^0}\beta_{i,j}\nabla g_{i,j}(z_{\lex}^\star)
 \right)\nonumber\\
&+\sum_{a\in\mathcal A}\lambda_a\nabla h_a(z_{\lex}^\star)
+\sum_{e\in\mathcal E}\mu_e\nabla q_e(z_{\lex}^\star),\nonumber
\end{align}
with $\lambda\ge0$ and $\mu$ free. Thus, the feasible set in
$(w,\beta,\lambda,\mu)$ is polyhedral; $\Omega(p^\star)$ is its projection onto
the $w$ coordinates. For piecewise-affine $J$, the same construction follows
after the standard epigraph lifting. Rank deficiency does not invalidate the
projection, but it can produce equality constraints and a lower-dimensional
region; it is therefore reported rather than silently removed.

The deployed inner problems use a convex quadratic tracking objective and
therefore do not satisfy Setting~(P). The following result retains the
polyhedral weight-region description without asserting that the upper image
itself is polyhedral.

\begin{proposition}[Polyhedral region under convex differentiable $J$]
\label{prop:convexJ}
Let Setting~\textup{(P$'$)} retain polyhedral $\mathcal Z$, affine
$g_{i,j}$, and $\ell_1$ hinges, while allowing $J$ to be convex and
differentiable near $\mathcal Z$ (in particular, convex quadratic). Under
\textup{(P$'$)} and \textup{(A1)--(A3)}, $\Omega(p^\star)$ is the
projection onto $w$ of the polyhedral scaled-KKT system
\eqref{eq:scaled-subgradient}--\eqref{eq:scaled-kkt}. Membership of any
fixed $w$ is decidable by one linear feasibility program. The
upper-image polyhedrality and vertex conclusions of
Theorem~\ref{thm:poly} need not hold.
\end{proposition}

\begin{proof}
For fixed $w$, convex first-order optimality gives
\begin{align}
z_{\lex}^\star&\in\argmin_{z\in\mathcal Z}
 \left[J(z)+\sum_iw_iV_i(z)\right]\\
&\iff
 0\in\nabla J(z_{\lex}^\star)
   +\sum_iw_i\partial V_i(z_{\lex}^\star)
   +N_{\mathcal Z}(z_{\lex}^\star).\nonumber
\end{align}

For a strictly violated hinge, the contribution is
$w_i\nabla g_{i,j}(z_{\lex}^\star)$; for a boundary-binding hinge it is
$\beta_{i,j}\nabla g_{i,j}(z_{\lex}^\star)$ with
$0\le\beta_{i,j}\le w_i$; and a strictly satisfied hinge contributes
zero. Since $N_{\mathcal Z}(z_{\lex}^\star)$ is generated by the active
polyhedral hard-constraint gradients, the complete condition is affine in
$(w,\beta,\lambda,\mu)$. Its feasible set is polyhedral, and its
projection onto $w$ is therefore polyhedral. By
Theorem~\ref{thm:main}, that projection is exactly $\Omega(p^\star)$.
Fixing $w$ leaves a linear feasibility program in
$(\beta,\lambda,\mu)$.
\end{proof}

This proposition justifies offline membership and calibration calculations
for convex-quadratic inner problems with $\ell_1$ hinges, provided the scaled
subgradient variables are used.

For the exact compliance vector
$b(z)=(\mathbb 1[V_1(z)=0],\ldots,\mathbb 1[V_L(z)=0])$ and a tolerance
$\boldsymbol\epsilon>0$, let
$b_{\boldsymbol\epsilon}(z)
=(\mathbb 1[V_i(z)\le\epsilon_i])_i$.

\begin{theorem}[$\ell_1$ exact compliance equivalence]\label{thm:l1comp}
Under Setting~\textup{(P)} and \textup{(A1)--(A3)}, any
$w\in\Omega(p^\star)$ satisfying
\[
 (w,1)\in
 \operatorname{relint}N_{\overline{\mathcal P}}^+(p^\star)
\]
exposes $p^\star$ as a singleton and therefore yields
$z_{\mathrm{ws}}^\star(w)=z_{\lex}^\star$. In particular, if the
unit-performance hyperplane intersects
$\operatorname{int}N_{\overline{\mathcal P}}^+(p^\star)$ inside the
positive orthant, then
$\operatorname{int}\Omega(p^\star)$ consists of exact-equivalence weights
and reproduces the exact compliance vector.
\end{theorem}

\begin{proof}
A direction in the relative interior of the normal cone of a polyhedron at
a vertex exposes that vertex as a singleton.
Therefore every WS minimizer $z$ satisfies $\Phi(z)=p^\star$. It follows
that $z\in\mathcal F_L$ and $J(z)=J^\star$; by~(A2),
$z=z_{\lex}^\star$. If the unit-performance hyperplane meets the interior
of the normal cone, the relative-interior intersection identity gives
\[
 \operatorname{int}\Omega(p^\star)
 =
 \{w:(w,1)\in
 \operatorname{int}N_{\overline{\mathcal P}}^+(p^\star)\},
\]
inside $\mathbb R_{++}^L$, proving the final statement.
\end{proof}

\subsection{\texorpdfstring{The $\ell_2$ Regime: Foreclosure and Rates}{The L2 Regime: Foreclosure and Rates}}
\label{sec:theory:l2}

The convex-quadratic \emph{penalty} differs structurally. In Setting~(Q),
each $V_i$ is continuously differentiable and a squared hinge supplies no
first-order force at its boundary.

\begin{lemma}[Vanishing boundary force under $\ell_2$]\label{lem:vanish}
Under Setting~\textup{(Q)}, if $g_{i,j}(z)=0$, then
\[
 \nabla\!\left([g_{i,j}(z)]_+^2\right)
 =2[g_{i,j}(z)]_+\nabla g_{i,j}(z)=0.
\]
In particular, $\nabla V_i(z_{\lex}^\star)=0$ whenever
$V_i^\star=0$.
\end{lemma}

\begin{theorem}[Local $\ell_2$ penalty branch]\label{thm:l2asymp}
Under Setting~\textup{(Q)} and the regularity conditions
\textup{(R1)--(R5)} of Appendix~\ref{app:l2}, let level $i$ satisfy
$V_i(z_{\lex}^\star)=0$, and let
$\lambda^\star(w_{-i})>0$ denote the limiting multipliers of its
boundary-binding hinges for $w_{-i}$ in a compact
$K\subset\mathbb R_{++}^{L-1}$. Then there is
$\overline w_i<\infty$ and, for every $w_{-i}\in K$ and
$w_i\ge\overline w_i$, a locally unique strict-local-minimizer branch
$z_{\mathrm{loc}}(w_i,w_{-i})$ near $z_{\lex}^\star$ such that, uniformly
over $K$,
\begin{equation}\label{eq:asymprate}
\begin{aligned}
 g_{i,j}\!\left(z_{\mathrm{loc}}(w_i,w_{-i})\right)
 &=\frac{\lambda_j^\star(w_{-i})}{2w_i}+O(w_i^{-2}),\\
 V_i\!\left(z_{\mathrm{loc}}(w_i,w_{-i})\right)
 &=\frac{\|\lambda^\star(w_{-i})\|_2^2}{4w_i^2}
   +O(w_i^{-3}).
\end{aligned}
\end{equation}
Moreover, if $\lambda^\star(w_{-i})\ne0$, then
$z_{\lex}^\star$ is not a stationary point of the finite-weight smooth
penalized problem for any $w_i<\infty$, and each component with
$\lambda_j^\star>0$ is strictly violated along the local branch for all
sufficiently large finite $w_i$.
\end{theorem}

Theorem~\ref{thm:l2asymp} is deliberately local: it constructs and
characterizes the minimizer branch near $z_{\lex}^\star$. Identification of
that branch with the \emph{global} WS minimizer requires an additional
global condition, such as global strong convexity of the penalized
objective or a separate argument showing that all global minimizers enter
the stated neighbourhood. Appendix~\ref{app:l2} gives the perturbation
argument; the single-binding-hinge closed form of the limiting multiplier
is derived in the supplementary material.

\begin{proposition}[$\ell_2$ tolerance bounds and scaling]\label{prop:l2tol}
Under the assumptions of Theorem~\ref{thm:l2asymp}, and assuming that
levels already violated at $z_{\lex}^\star$ are separated from their
compliance thresholds by a positive margin, there exist uniform constants
$C_{g,i},C_{V,i}\ge0$ and $\overline w_i<\infty$ such that, for
$w_{-i}\in K$ and $w_i\ge\overline w_i$,
\begin{align}
 \max_{j\in\mathcal B}
 \left|g_{i,j}(z_{\mathrm{loc}})\right|
 &\le
 \frac{\|\lambda^\star(w_{-i})\|_\infty}{2w_i}
 +\frac{C_{g,i}}{w_i^2},
 \label{eq:l2-raw-bound}\\
 V_i(z_{\mathrm{loc}})
 &\le
 \frac{\|\lambda^\star(w_{-i})\|_2^2}{4w_i^2}
 +\frac{C_{V,i}}{w_i^3}.
 \label{eq:l2-V-bound}
\end{align}
Consequently, sufficient finite thresholds are
\begin{align}
 w_i&\ge
 \max\left\{
 \overline w_i,\,
 \frac{\|\lambda^\star(w_{-i})\|_\infty}{\epsilon_i},\,
 \sqrt{\frac{2C_{g,i}}{\epsilon_i}}
 \right\}
 &&\text{for }|g_{i,j}|\le\epsilon_i,\nonumber\\
 w_i&\ge
 \max\left\{
 \overline w_i,\,
 \frac{\|\lambda^\star(w_{-i})\|_2}{\sqrt{2\epsilon_i}},\,
 \left(\frac{2C_{V,i}}{\epsilon_i}\right)^{1/3}
 \right\}
 &&\text{for }V_i\le\epsilon_i.\nonumber
\end{align}
As $\epsilon_i\downarrow0$, the leading scalings are $\Theta(\epsilon_i^{-1})$ for raw-constraint tolerance and $\Theta(\epsilon_i^{-1/2})$ for violation-functional tolerance.
\end{proposition}

The KKT stationarity equation for $(\mathrm P_\infty)$ shows that the
limiting multiplier is affine in $w_{-i}$ when $G$ has full row rank.
Omitting the remainder terms gives
the useful first-order candidates
$\|\lambda^\star\|_\infty/(2\epsilon_i)$ and
$\|\lambda^\star\|_2/(2\sqrt{\epsilon_i})$, but these candidates alone are
not finite certificates unless the remainder is zero, bounded explicitly,
or checked by resolving the WS problem. Example~2 below is a special
closed-form case in which the candidate is exact.

\subsection{Robustness Across Decision Instances}
\label{sec:theory:robust}

A receding-horizon controller does not solve one convex problem but a new
one each tick, so the deployable object is a weight valid on a whole set of
them.

\begin{proposition}[Robust equivalence region]\label{prop:robust}
Let inner problems $q=1,\dots,K$ each satisfy the assumptions of
Theorem~\ref{thm:main}, with lexicographic points $p^\star_q$, and define the
\emph{robust region}
\begin{equation}\label{eq:robust-region}
  \Omega_{\mathrm{rob}}:=\bigcap_{q=1}^{K}\Omega(p^\star_q).
\end{equation}
If $w\in\Omega_{\mathrm{rob}}$, then for every $q$ the lexicographic optimum of
problem $q$ minimises its weighted-sum objective at weight $w$. Moreover,
$\Omega_{\mathrm{rob}}$ is a polyhedron whenever each $\Omega(p^\star_q)$ is ---
by Theorem~\ref{thm:poly} for piecewise-affine $J$ and by
Proposition~\ref{prop:convexJ} for convex differentiable $J$ --- and its
Chebyshev centre is then a single linear program once the half-space
descriptions are available.
\end{proposition}
\begin{proof}
The first claim is Theorem~\ref{thm:main} applied to each $q$ separately; the
second follows because polyhedra are closed under finite intersection.
\end{proof}

\begin{remark}
The guarantee is exact on the $K$ certified problems; extension to unseen
ticks is an empirical question, measured in
Section~\ref{sec:experiments:certificate}.
\end{remark}

\emph{Persistence interpretation.} Pointwise membership
$w\in\Omega(p^\star_q)$ certifies one decision instance; persistence over
MPC ticks requires a nonempty intersection of their regions and is
measured in Section~\ref{sec:experiments:certificate}.

\begin{remark}[Regularity and applicability]\label{rem:applicability}
\phantomsection\label{sec:theory:failure}%
Two conditions bound the guarantees, and both are testable. If augmented
LICQ fails at $z_{\lex}^\star$, the region may be lower-dimensional and
direct affine elimination is unavailable; this is certified by a
singular-value rank calculation on the augmented active-gradient matrix. If
$\mathcal Z$, any $g_{i,j}$, or $J$ is non-convex, the support-normal
characterization no longer provides a global WS guarantee; convexity of the
encoders is verified directly.
\end{remark}

\section{Construction, Certificate, and Verification}
\label{sec:construction}
Section~\ref{sec:equivalence_region} establishes constructive results. Here we present the calibration algorithm, its computational complexity, the online monitoring procedure, and the numerical verification.

\subsection{Constructive Calibration and Online Deployment}
\label{sec:method:algorithms}

We compute the equivalence region offline for a reference convex problem. When the intersection with the operator box is non-empty, we condense the region into a robust interior weight $w^\dagger$. Algorithm~\ref{alg:calib} accepts an operator weight box $[\underline{w}_i,\overline{w}_i]$, projects the scaled KKT system to obtain $(C,d,E,f)$, and returns the Chebyshev centre of $\Omega(p^\star)$ within the box, together with the inscribed-ball radius $r^\dagger$.

\begin{algorithm}[t]
\caption{Calibration: $\ell_1$ exact equivalence region}
\label{alg:calib}
\begin{algorithmic}[1]
\Require Priority-ordered problem with $\ell_1$ penalties; operator weight box $[\underline{w}_i,\overline{w}_i]$ with $\underline w_i>0$. \State \textbf{Lex solve.} Solve the cascade for $z_{\lex}^\star$ and $p^\star=\Phi(z_{\lex}^\star)$. \State \textbf{Classify hinges.} Form $\mathcal I_i^+$, $\mathcal I_i^0$, and the strictly satisfied set; evaluate $\nabla J(z_{\lex}^\star)$ and the active gradients.
\State \textbf{Scaled KKT assembly.} For every $(i,j)$ with $j\in\mathcal I_i^0$, introduce $\beta_{i,j}=w_i\alpha_{i,j}$ and impose $0\le\beta_{i,j}\le w_i$. Form~\eqref{eq:scaled-kkt} as an affine system in $(w,\beta,\lambda,\mu)$.
\State \textbf{Projection.} Project out $(\beta,\lambda,\mu)$ to obtain
$\Omega(p^\star)=\{w:Cw\le d,\ Ew=f,\ w\in\mathbb R_{++}^L\}$.
\State \textbf{Chebyshev centre.} Solve
$\max_{w,r}\,r$ subject to
$c_k^\top w+r\|c_k\|_2\le d_k$,
$Ew=f$,
$w_i-\underline{w}_i\ge r$,
$\overline{w}_i-w_i\ge r$, and $r\ge0$.
\State \textbf{Verify.} Solve the WS problem at $w^\dagger$ and verify
$\Phi(z_{\mathrm{ws}}^\star(w^\dagger))=p^\star$ to solver tolerance;
when the WS minimizer is unique, also verify
$z_{\mathrm{ws}}^\star(w^\dagger)=z_{\lex}^\star$ and the compliance-vector match.
\Ensure Weight $w^\dagger$; margin $r^\dagger$; description $(C,d,E,f)$; numerical verification record.
\end{algorithmic}
\end{algorithm}

A homogeneous-cone variant omits the operator box and normalizes $\widehat{\Omega}(p^\star)$ on a simplex slice. The construction appears in the supplementary material.

The homogeneous exact-equivalence construction for $\ell_2$ penalties fails when a binding constraint admits a positive limiting multiplier. Proposition~\ref{prop:l2tol} provides local asymptotic bounds and first-order weight scaling in this case. The $\ell_2$ procedure computes a first-order candidate from the limiting multipliers and solves the WS problem to verify the prescribed tolerance. The procedure yields a finite certificate only when the remainder constants in Proposition~\ref{prop:l2tol} are included. Otherwise, it constitutes a candidate-and-check method rather than an exact single-LP guarantee. If the $\ell_1$ box admits no positive-margin intersection, the routine returns a flagged candidate with $r^\dagger=0$. Such a return does not constitute an equivalence certificate. Box expansion, recalibration, or cascade fallback is then required.

During online deployment, the planner solves the WS problem~\eqref{eq:ws} at the cached weight for the current convex inner problem, forms the binary compliance vector, and compares it with the cached calibration pattern. This comparison serves as a computational proxy for certificate change; it does not constitute an exact membership test, which would require solving the current lexicographic reference and verifying its scaled-KKT system. The intersection of the current and calibration-instance regions, as characterized in Proposition~\ref{prop:robust}, can vanish when the active constraints or rule applicability change. Consequently, a proxy mismatch initiates cascade fallback.
Section~\ref{sec:experiments:certificate} measures certificate persistence on held-out instances, and Section~\ref{sec:calibrated:conditions} evaluates the proxy against exact cross-tick membership.

\emph{Complexity.}
General polyhedral projection is exponential in the number of eliminated variables. Membership of a fixed $w$ is checked by a polynomial-time linear feasibility problem in $(\beta,\lambda,\mu)$. Direct affine elimination applies under augmented LICQ: the boundary-hinge gradients together with the active hard-inequality and equality gradients in~\eqref{eq:scaled-kkt} are linearly independent. Under this condition, $(\beta,\lambda,\mu)$ is uniquely determined affinely by $w$. Substituting this expression into
$0\le\beta_{i,j}\le w_i$ and $\lambda\ge0$ yields $O(A)$ linear inequalities, where $A$ denotes the active-set size. Without augmented LICQ, the projection remains valid but may be lower-dimensional and may require a general projection method. Once an explicit half-space representation is available, the Chebyshev-center problem has $L+1$ decision variables.

\subsection{Worked Examples and Numerical Verification}
\label{subsec:verification}

Two worked examples substantiate the theory and the reference implementation. Figure~\ref{fig:certificate_geometry} draws the central object for Example~1.

\begin{figure}[t]
 \centering
 \includegraphics[width=\linewidth]{
 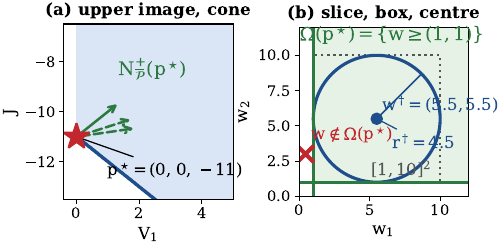}
 \caption{The certificate, drawn (Example~1, $L=2$). (a)~Cross-section of  the upper image at $V_2=0$: the lexicographic point $p^\star$ is a vertex,  and the weights that recover it are the normal directions there. (b)~The  unit-performance slice $\Omega(p^\star)=\{w\ge(1,1)\}$, the operator box  $[1,10]^2$, and the Chebyshev centre $w^\dagger=(5.5,5.5)$ with inscribed radius $r^\dagger=4.5$.}
 \label{fig:certificate_geometry}
\end{figure}

\emph{Example~1 ($\ell_1$ LP).}
Let $\mathcal Z=[0,10]^2$, $\mathcal C_1:z_1+z_2\le8$, $\mathcal C_2:z_1\le3$, and $J=-2z_1-z_2$. The cascade gives $z_{\lex}^\star=(3,5)$ and $p^\star=(0,0,-11)$. Both rule constraints are boundary-binding. The scaled coefficients satisfy $0\le\beta_i\le w_i$, and the stationarity system yields $\Omega(p^\star)=\{(w_1,w_2)\in\mathbb R_{++}^2:w_1\ge1,\,w_2\ge1\}$, with Chebyshev centre $w^\dagger=(5.5,5.5)$ and margin $r^\dagger=4.5$ on $[1,10]^2$.

\emph{Example~2 ($\ell_2$ QP).}
Let $x_{k+1}=x_k+u_k$, $J=(x_4-10)^2$, with three squared-hinge levels: safety $x_k\le3$, legal $u_k\le2$, and comfort $u_k\ge1.5$. The cascade gives $u_{\lex}^\star=(0.75,0.75,0.75,0.75)$ and $J^\star=49$. The
closed-form penalized minimizer lies on the branch of Theorem~\ref{thm:l2asymp}; the closed-form derivation, given in the supplementary material, shows that the $V_1\le0.01$ threshold at $w_3=1$
is exactly $w_1\ge77.5$ --- exact here because the remainder vanishes, a property special to this example.

\begin{figure}[t]
\centering
\includegraphics[width=\linewidth]{
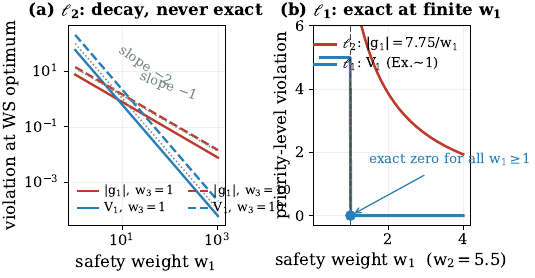}
\caption{The $\ell_1$/$\ell_2$ penalty dichotomy. \textbf{(a)} For Example~2, the raw $\ell_2$ boundary violation decays as $O(1/w_1)$ and its square as $O(1/w_1^2)$; neither reaches zero at finite weight. \textbf{(b)} In Example~1, the $\ell_1$ priority violation reaches exactly zero at the finite threshold $w_1=1$ and remains zero thereafter.}
\label{fig:l1l2_decay}
\end{figure}

A dynamics-agnostic reference implementation reproduces both examples to $10^{-6}$ and implements the scaled-subgradient formulation $\beta_{i,j}=w_i\alpha_{i,j}$. $10^{-6}$ and implements the scaled subgradient formulation $\beta_{i,j}=w_i\alpha_{i,j}$. Its randomized harness uses $500$ instances per objective family (linear $J$ and convex-quadratic $J$), random polyhedral $\mathcal Z$, and affine hinges. Region predictions are compared with brute-force weighted solves at the Chebyshev center, on both sides of each computed boundary, and at random box draws, using trajectory equality rather than only compliance-bit agreement. Across the $1000$ randomized instances, the scaled-KKT predictions agree with the direct-solve oracle at every decisive test point.

\section{Experimental Design}
\label{sec:deployment}

The closed-loop nuPlan study~\cite{caesar2021nuplan} has two layers: a dynamics-agnostic calibration layer that computes, from active rule-encoder gradients and hinge classifications at a lexicographic reference solution and using the scaled variables $\beta_{i,j}=w_i\alpha_{i,j}$, an $\ell_1$ equivalence region, a robust weight $w^\dagger$, and a cached compliance pattern offline; and a planning layer (kinematic-bicycle dynamics, sequential linearisation, rule encoders, cascade solver, planner-independent observer). Each sequentially linearised inner problem is a convex decision instance to which the support-normal guarantees apply when its assumptions and active-set certificate hold; with the convex quadratic tracking objective, Proposition~\ref{prop:convexJ} supplies the scaled-KKT weight-region description. The comparative campaign characterizes the heuristic surrogate at $w_{\mathrm{dep}}$ (Section~\ref{sec:system:params}); the calibrated study of Section~\ref{sec:experiments:calibrated} deploys the Chebyshev-calibrated weight under the monitor.

A single design decision governs the interpretation of all results. The study comprises two experiments: the comparative campaign (Section~\ref{sec:experiments}) and the calibrated study (Section~\ref{sec:experiments:calibrated}). Every reported contrast is computed within a single experiment.
Each experiment employs identical per-tick problem construction, seeds, and scenario instances, and maintains its own heuristic and cascade reference conditions. The two constructions differ in a single element. The comparative campaign derives the reachable velocity cap from the reference path. The calibrated study clamps the velocity cap to values attainable from the measured initial state. This modification removes spurious first-step infeasibility in cascade stage problems and ensures that stage failures fail closed. Cross-experiment comparisons are not reported. Each experiment's conclusions are internal contrasts under its own construction.

\subsection{The 25-Rule Rulebook and Its Partition}
\label{sec:system:rulebook}

The deployment implements twenty-five nuPlan-applicable traffic rules as described by Censi~\etal~\cite{censi2019}, with numeric priorities $0$--$10$ organized into three sequentially enforced tiers --- safety ($L_{\mathrm{S}}$), legal ($L_{\mathrm{L}}$), comfort ($L_{\mathrm{C}}$); $L=3$, with desired-speed tracking in $J$. The rules partition $13{+}9{+}3$ by planner action: the MPC-controlled class (thirteen rules, twelve encoder modules; full rulebook in the supplementary material) utilizes active convex encoders that emit per-tick affine constraints; the observer-only class (nine rules, including the $H_2$ sentinels \texttt{8r0}, \texttt{8r1} at $L_8$ and \texttt{2r2} at $L_{2}$) is monitored without an encoder and serves as a negative-control set; the stub class (three rules) reserves inactive capacity. Controller-side claims are restricted to this three-tier, thirteen-rule implementation.

\subsection{The Two-Level Pipeline}
\label{sec:system:pipeline}

A two-level pipeline separates global routing from per-tick control. The outer level maintains a reference corridor from the lane graph. The inner level performs per-tick optimal control at $10$ Hz over a $3.0$ second horizon with $N=30$ and $\Delta t=0.1$ s. A monolithic single-level problem spanning the entire route window exceeds the $100$ ms per-tick computational budget.

Each tick runs a sequential-linearisation (SLP) outer loop. The kinematic bicycle $\dot p=v\cos\psi,\ \dot q=v\sin\psi,\ \dot\psi=(v/\ell)\tan\delta$ is non-convex in the heading; the planner recovers convexity locally by linearising about the warm start $\tilde z^{(j)}$ at each iteration~$j$,
\begin{align}\label{eq:slp-iter}
 \tilde z^{(j+1)}&=\argmin_{z}\;J(z)+\sum_i w_i V_i(z)\\
 &\text{s.t.}\quad x_{k+1}=A_k^{(j)}x_k+B_k^{(j)}u_k+c_k^{(j)},\nonumber
\end{align}
re-solving until the residual $\rho^{(j)}=\max_k\|\tilde z_k^{(j+1)}-\tilde z_k^{(j)}\|_2$ falls below
$0.05$ m or the iteration cap is reached. The cap is $1$ in $C_1$ and $3$ in $C_2$. This is the single parameter distinguishing the two protocol conditions. The convex inner problem is
\begin{equation}\label{eq:lcp-ocp}
 \min_{\tilde x,\tilde u,T}\ J(\tilde x,\tilde u)+\sum_{i=1}^{L}w_i V_i(T_i)+\eta\sum_{i,j,k}T_{i,j,k}^{2},
\end{equation}
subject to the linearised dynamics, actuator and slew-rate bounds, and the per-tier epigraph constraints
\begin{equation}\label{eq:lcp-epigraph}
 a_{i,j,k}^{\top}\tilde x_k+b_{i,j,k}^{\top}\tilde u_k+e_{i,j,k}\,\mu_{i,j,k}\le T_{i,j,k},\quad T_{i,j,k}\ge0,
\end{equation}
where $(a_{i,j,k},b_{i,j,k},e_{i,j,k})$ is the affine coefficient emitted by the rule encoder at tier $i$, slot $j$, step $k$; the binary mask $\mu_{i,j,k}\in\{0,1\}$ gates inapplicable rules; and the Tikhonov parameter is $\eta=10^{-6}$. This value improves the reduced-Hessian conditioning of the interior-point solve. In the equivalence analysis, the term is absorbed into the differentiable convex performance term. Its gradient with respect to a boundary slack is zero at $T_{i,j,k}=0$. It can affect solutions with positive slack.

The planner operates in a dual weighted-sum and cascade mode. A single weighted-sum solve is performed online. The full cascade is available for the offline reference and as the runtime fallback. A runtime compliance check compares the realized compliance vector against the cached calibration pattern. A planner-independent observer monitors the deployment. The observer evaluates the complete $25$-rule rulebook at every tick using venue-agnostic predicates over the simulator state. For example, the observer computes footprint-overlap area for a collision rule. The observer aggregates the per-rule integrated violation $\int r_i(t)\,\mathrm{d}t$. The observer shares no encoder code with the planner and supplies the primary rule-violation record used by the campaign. The remaining engineering apparatus is recorded in the supplementary material.

\subsection{Independent Compliance Certification}
\label{sec:system:certification}

Because the observer shares authorship with the planner, a second, independently authored evaluator, Construe, is employed for external corroboration. This offline certificate engine implements twenty-seven rules from the same rulebook using independently developed interpretations and code, evaluates every trajectory without shared rule code, and is accessed read-only under a verified content manifest. The two rule sets are joined through an explicit taxonomy, excluding six identically named but semantically distinct identifiers. Certificate outputs are not used in the per-level violation vectors or dominance statistics. Cross-engine agreement is calibrated separately for each rule pair on golden logs; only pairs promoted by that calibration support agreement claims. As reported in Section~\ref{sec:campaign:order}, one pair---vehicle bounding-box overlap---is promoted to verified status, so the independent evaluator corroborates that result rather than the complete rulebook.

\subsection{Deployment Parameters and What the Campaign Ran}
\label{sec:system:params}

The experimental conditions (Table~\ref{tab:conditions}) utilize a common parameter set and differ only in penalty form ($\ell_1$ or $\ell_2$), runtime mode (weighted sum or full cascade), and SLP cap ($1$ or $3$). Two parameters enter the analysis directly --- the per-tier tolerance $\boldsymbol{\epsilon}=(10^{-4},4{\times}10^{-2},5{\times}10^{-1})$ and the heuristic weights below; the remaining actuator, tracking, and vehicle parameters are listed in the supplementary material. Each protocol cell is deterministic given its seed, and the analysis pipeline that processes the per-tick logs into reported quantities is provided as a reproducibility artifact.

The campaign utilized priority-scaled heuristic weights
$w_{\mathrm{dep}}=(1000,100,10)$ under $\ell_1$ and
$(10^4,10^3,100)$ under $\ell_2$. The $C_1$ to $C_4$ results characterize the surrogate at $w_{\mathrm{dep}}$ and not the Chebyshev-center method. The calibrated and guarded conditions of Section~\ref{sec:experiments:calibrated} are reported separately and use a weight calibrated on held-out logs. For each scenario class, the first-tick exemplar cascade is re-solved at tightened tolerance $10^{-8}$. The scaled-subgradient KKT feasibility system \eqref{eq:scaled-subgradient} to \eqref{eq:scaled-kkt} is assembled from the encoder gradients. The system has $184$ dimensions and up to $190$ hinges. Membership $w_{\mathrm{dep}}\in\Omega(p^\star)$ is tested by linear feasibility in $(\beta,\lambda,\mu)$. When many comfort hinges are violated at the lex optimum, exact trajectory equivalence is frequently unavailable, and tolerance-compliance becomes the relevant target. The post-hoc audit identifies this situation in twelve of the sixteen classes. Among the remaining four, $w_{\mathrm{dep}}$ lies in the computed region for one class. In two classes with broad non-empty regions, the audit gives a comfort-tier lower bound of $89$, while $w_{\mathrm{dep}}$ supplies $10$. The numerical audit values reported here are obtained from this scaled-variable feasibility system.

\subsection{Held-Out Certificate-Persistence Protocol}
\label{sec:experiments:persistence-protocol}

Certificate persistence is evaluated on three \texttt{nuPlan-mini} logs disjoint from the nine the closed-loop campaign draws on: convex inner problems are captured every $25$ controller ticks ($2.5$\,s), giving $84$
exemplars across $14$ scenario classes, each re-solved by the cascade at tightened tolerance ($10^{-8}$) with its equivalence region constructed from the scaled-KKT system~\eqref{eq:scaled-subgradient}--\eqref{eq:scaled-kkt}. An exemplar is \emph{eligible} when the boxed region has strictly positive inscribed
radius, genuine non-box facet margin, and relative stationarity residual at most $10^{-6}$; failure is classified as direct-elimination failure under augmented-LICQ deficiency (Remark~\ref{rem:auglicq}), infeasibility within the operator box, or failure of the lexicographic reference solve itself.

Three persistence metrics are reported: robust-intersection survival $S_k=\mathbf 1\bigl[\bigcap_{q=1}^{k}\Omega(p^\star_q)\neq\varnothing\bigr]$, whose largest $k$ with $S_k=1$ is the class's certificate \emph{lifetime}; the cross-tick membership rate
\begin{equation}\label{eq:cross-tick-rate}
 M_{\mathrm{cross}}
 =\frac{\sum_{q\neq r}\mathbf 1\bigl[w^\dagger_q\in\Omega(p^\star_r)\bigr]}
        {\#\{(q,r):q\neq r\}},
\end{equation}
with $w^\dagger_q$ the weight calibrated at exemplar $q$; and active-set churn, the fraction of a class's exemplars with distinct classified binding sets. To resolve the $2.5$\,s sampling interval, one held-out
episode is additionally captured at every $0.1$\,s controller tick; the dense study measures the number of subsequent ticks for which a certificate issued at tick $t$ remains valid.

\subsection{Protocol Design and Empirical Objects}
\label{sec:experiments:protocol}

The comparative campaign assesses five planner conditions across sixteen \texttt{nuPlan-mini} scenario types and five seeds, yielding $5\times5\times16=400$ closed-loop cells (scenario table in the supplementary material). All conditions (Table~\ref{tab:conditions}) employ the same pipeline, encoders, warm start, and hyperparameters; only the penalty form, runtime mode, and SLP cap vary. $C_0$ is the legacy flat-weight baseline; $C_1$ the single-solve $\ell_1$ research condition at the heuristic deployed weights; $C_2$ and $C_3$ the SLP-cap and penalty-form ablations; $C_4$ the full $L{+}1$-stage lex cascade per tick, the reference for comparison with $C_1$.

\begin{table}[t]
 \centering
 \caption{Planner conditions by experiment. Conditions within an experiment  share solver settings, seeds, and scenario instances, and every contrast is computed within one experiment. Weights: heur.\ =
 $w_{\mathrm{dep}}=(1000,100,10)$; cal.\ = frozen held-out cache
 (Table~\ref{tab:calibration_record}).}
 \label{tab:conditions}
 \scriptsize
 \setlength{\tabcolsep}{3pt}
 \begin{tabular}{lllrl}
 \toprule
 Code & Solve & Weights & Cells & Role \\
 \midrule
 \multicolumn{5}{@{}l}{\emph{Comparative campaign} (Section~\ref{sec:experiments})}\\
 $C_0$ & flat single-tier & flat & $16{\times}5$ & legacy baseline \\
 $C_1$ & WS-$\ell_1$, SLP-1 & heur. & $16{\times}5$ & heuristic single solve \\
 $C_2$ & WS-$\ell_1$, SLP-3 & heur. & $16{\times}5$ & $H_3$ ablation \\
 $C_3$ & WS-$\ell_2$, SLP-1 & heur. & $16{\times}5$ & $H_4$ ablation \\
 $C_4$ & cascade-$\ell_1$ & --- & $16{\times}5$ & cascade reference \\
 \midrule
 \multicolumn{5}{@{}l}{\emph{Calibrated study} (Section~\ref{sec:experiments:calibrated})}\\
 $C_H$ & WS-$\ell_1$, SLP-1 & heur. & $11{\times}5$ & heuristic reference \\
 $C_5$ & WS-$\ell_1$, SLP-1 & cal. & $11{\times}5$ & calibrated single solve \\
 $C_6$ & WS-$\ell_1$ + monitor & cal. & $11{\times}5$ & guarded (cascade fallback) \\
 $C_X$ & cascade-$\ell_1$ & --- & $11{\times}5$ & cascade reference \\
 \midrule
 \multicolumn{5}{@{}l}{\emph{External baseline} (supplementary material)}\\
 IDM & external planner & --- & $16{\times}5$ & external comparison \\
 \bottomrule
 \end{tabular}
\end{table}

Three metrics derived from observer evaluations form the basis of the analysis. The \emph{per-level integrated violation} at numeric priority level $\ell$ on a cell $(c,\sigma,s)$ is \begin{equation}\label{eq:V-level} V_{\mathrm{MPC},\ell}^{(c,\sigma,s)}:=\sum_{i\in\mathcal{R}_{\ell}^{\mathrm{MPC}}}\sum_{t}r_{i,t}^{(c,\sigma,s)}\,\Delta t,
\end{equation}
where $\mathcal{R}_{\ell}^{\mathrm{MPC}}$ is the set of MPC-controlled rule IDs at level $\ell$, $r_{i,t}\in[0,1]$ is the observer's per-tick violation rate, and $\Delta t=0.1$\,s; a planner that respects the hierarchy concentrates $V_{\mathrm{MPC},\ell}$ at the low-priority levels and holds it near zero at the high ones.

\begin{definition}[Lex-Pareto dominance]\label{def:lex-pareto}
 Condition $c_M$ lex-Pareto-dominates $c_B$ on scenario $s$, written  $c_M\succ_{\lex}c_B$, when the per-level violation vectors  $\mathbf{V}^{(c,\cdot,s)}=(V_{\mathrm{MPC},L}^{(c,\cdot,s)},\ldots,V_{\mathrm{MPC},0}^{(c,\cdot,s)})$,  ordered highest to lowest priority and evaluated on the per-seed median,  satisfy lexicographic precedence within a per-level tolerance  $\boldsymbol{\tau}>0$: at the highest level $\ell^\dagger$ at which the two  differ by more than $\tau_{\ell^\dagger}$,  $V_{\mathrm{MPC},\ell^\dagger}^{(c_M,\cdot,s)}<V_{\mathrm{MPC},\ell^\dagger}^{(c_B,\cdot,s)}-\tau_{\ell^\dagger}$,  with the two within tolerance at every strictly higher level.
\end{definition}

This hierarchical comparison does not necessitate improvement at every level; the cascade may allocate lower-priority resources to enhance higher-priority levels. The dominance count represents the per-condition lex-win rate. Fidelity of the research condition to the cascade is quantified by the $C_1$-vs-$C_4$ \emph{per-tick compliance match}
\begin{equation}\label{eq:compliance-match}
 M_\ell^{(s)}:=\operatorname{median}_\sigma\frac{1}{T_s}\sum_{t=1}^{T_s}\mathbb{1}\bigl[b_{\epsilon,\ell,t}^{(C_1,\sigma,s)}=b_{\epsilon,\ell,t}^{(C_4,\sigma,s)}\bigr],
\end{equation}
where $b_{\epsilon,\ell,t}=\mathbb{1}[v_{\ell,t}\le\epsilon_\ell]$ is the per-tick compliance bit and the per-tick level sum $v_{\ell,t}$ runs over the MPC-controlled rules of the level, or, for the observer-only sentinel levels $\ell\in\{8,2\}$, over $\mathcal{R}_8^{\mathrm{obs}}=\{\texttt{8r0},\texttt{8r1}\}$ and $\mathcal{R}_2^{\mathrm{obs}}=\{\texttt{2r2}\}$. A value $M_\ell^{(s)}=1$ means that $C_1$ and $C_4$ have the same compliance bit at level $\ell$ on every evaluated tick of scenario $s$. Because violations are sparse relative to episode length, tick-level agreement can saturate; the event-level analysis below is therefore treated as the more sensitive fidelity assessment for sparse rule violations.

The four planned hypotheses are: $H_1$, the cascade $C_4$ lex-Pareto-dominates the legacy baseline $C_0$ on a majority of scenarios, by an exact one-sided binomial test with rejection threshold $k\ge12$ at $n=16$, $\alpha=0.05$; $H_2$, the research condition $C_1$ matches the cascade's per-tick compliance at the collision levels $L_{10},L_9$ and the observer-only sentinels $L_8,L_{2}$ against the floor $M_\ell^{(s)}\ge0.95$; $H_3$, raising the SLP cap from one to three improves the research condition's behaviour; and $H_4$, the $\ell_2$ penalty matches $\ell_1$ at the campaign tolerance. Per-level percent reductions carry $95\%$ BCa bootstrap intervals ($10^4$ resamples) across the sixteen scenarios; all statistics span the complete five-seed campaign. The binomial model is a fixed-benchmark sign test with exchangeable paired comparisons; at one-sided $\alpha=0.05$ the exact rejection threshold is $k\ge12$ ($\Pr(X\ge12\mid p{=}\tfrac12)\approx0.0384$).

\section{Calibration, Certificate Lifetime, and Closed-Loop Evaluation}
\label{sec:experiments:calibrated}

\subsection{Calibration Coverage and Certificate Lifetime}
\label{sec:experiments:certificate}

\emph{Coverage.} Under the protocol of Section~\ref{sec:experiments:persistence-protocol}, eleven of fourteen examined scenario classes contain an eligible held-out calibration exemplar --- a positive-margin certificate with relative stationarity residual at most $2.2\times10^{-7}$ (Table~\ref{tab:calibration_record}).
Three failures arise from distinct mechanisms. One exemplar, near\_multiple\_vehicles, violates augmented LICQ as defined in Remark~\ref{rem:auglicq}; consequently, the direct-elimination implementation withholds certification for the possibly lower-dimensional region. Two exemplars, starting\_protected\_cross\_turn and traversing\_intersection, admit no feasible weight within the operator box $[1,10^8]^3$.
Later exemplars from these two episodes certify with large margin. Re-running the same construction on the campaign's own first-tick exemplars yields ten of sixteen, with two such failures and four classes at which the lexicographic reference itself fails closed. These four classes certify without incident on the held-out states.
Certifiability is a property of the decision instance rather than the scenario class. The residuals reported in Table~\ref{tab:calibration_record} for the three failures correspond to the box-center fallback that the protocol reports on failure, not to a failed solve. The supplementary material provides the full failure diagnostics: the singular-value rank test, conditioning, facet counts, and later-tick recovery.

\emph{Sensitivity.} Eligibility is stable under a four-decade sweep of the hinge-classification band and under operator boxes $[1,10^4]^3$, $[1,10^6]^3$, and $[1,10^8]^3$. The normalized weight direction of every jointly eligible class moves by at most $0.5^\circ$, with median $0^\circ$, across the four decades. The center magnitude scales with the box where the region is broad. The verified stationarity residual remains below $10^{-6}$ at every box. The supplementary material reports per-threshold and per-box counts.

The robust region defined in equation~\eqref{eq:robust-region} is empty for all fourteen classes. No single weight reproduces the cascade across the sampled ticks of any episode.
The median lifetime of the intersection, grown one exemplar at a time, is one sampled interval with range 0 to 3. A weight calibrated at one tick satisfies the stationarity system of a different tick in 85 of 219 cross-tick pairs.
The cross-tick validity is $M_{\mathrm{cross}}=0.39$. Across the fourteen classes, the association is consistent with active-set turnover as the mechanism. The twelve classes whose binding set differs at every exemplar have cross-tick validity at most $0.40$. The two classes whose binding set is stable are valid at every tick. The Mann--Whitney test yields $p=0.017$. The Spearman correlation between active-set churn and cross-tick validity is $\rho=-0.61$ with $p=0.020$ for $n=14$ classes. See Fig.~\ref{fig:certificate_lifetime}.

\emph{Dense capture.} Dense recapture of one episode of \texttt{following\_lane\_with\_slow\_lead} at the native $0.1$\,s rate sharpens the result. Of $144$ usable ticks ($150$ captured; six dropped where a cascade stage fails closed), $83$ admit a pointwise certificate, but the median certificate remains valid for \emph{zero} subsequent ticks (mean $0.16$, maximum $3$, i.e.\ $0.3$\,s); only $9$ of $82$ adjacent-tick certificates remain valid, a rate of $0.11$. No certificate survives five ticks. The classified active set changes between $140$ of $143$ consecutive captures. The robust intersection over the dense episode is empty. The $2.5$\,s study's report of loss within one interval is not an artifact of coarse sampling. The dense result concerns one episode of one class. The remaining classes are resolved only at $2.5$\,s intervals.

These measurements do not imply that robust intersections are universally empty. They show that certificates are empirically short-lived in the evaluated receding-horizon instances and that persistence is strongly associated with active-set stability. The $\ell_1$ theory remains unchanged: the certificate is exact where it is issued. The measurement delimits its extent. This mechanism explains the tickwise and event divergence measured in the comparative campaign in Section~\ref{sec:experiments} and motivates the guarded deployment evaluated next.

\begin{figure}[t]
 \centering
 \includegraphics[width=\linewidth]{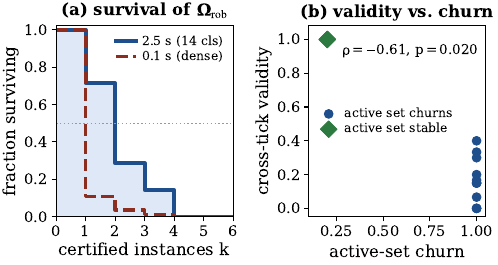}
 \caption{Certificate persistence on held-out logs. (a)~Survival of the robust-region intersection as additional decision instances are included:  the coarse multi-class study (solid) has median lifetime one $2.5$\,s  interval, while the dense $0.1$\,s recapture of one episode (dashed) has median lifetime zero subsequent ticks. (b)~Cross-tick certificate validity decreases with active-set churn ($\rho=-0.61$, $p=0.020$); the two classes with stable active sets remain valid across all sampled ticks.}
 \label{fig:certificate_lifetime}
\end{figure}

\subsection{Calibrated and Guarded Conditions}
\label{sec:calibrated:conditions}

The calibrated study evaluates four conditions on identical scenario--seed cells: a heuristic-weight single solve $C_H$ (weights $w_{\mathrm{dep}}$); the calibrated single solve $C_5$; the guarded condition $C_6$, which applies the runtime monitor of Section~\ref{sec:deployment} to the calibrated solve with cascade fallback; and the stage-wise lexicographic cascade $C_X$, solved to solver tolerance as the study reference. Solver conditioning under the calibrated weights, with components up to $10^8$, is unremarkable. Median IPOPT iteration counts are close between arms, with zero non-success returns in either. Median and tail quantiles are reported in the supplementary material. The monitor compares the realized compliance pattern against the cached calibration-state pattern. This is a heuristic proxy for certificate loss, not a per-tick membership certificate, which would require re-solving the cascade at the current instance.
Exact cross-tick KKT membership, as defined in Section~\ref{sec:experiments:certificate}, provides ground truth for evaluating this proxy. Pooled over the 148 evaluable pairs of the ten calibration-eligible classes with computable patterns, the monitor detects certificate lapse with sensitivity 0.55 and specificity 0.83. Predictive values are reported in the supplementary material. The miss rate is heterogeneous across classes, ranging from every lapse caught in changing\_lane\_to\_left, 16 of 16, to none in following\_lane\_with\_slow\_lead, 0 of 16. The monitor is therefore a selective fallback trigger, not an exact expiry detector, and the fallback-rate accounting below inherits this operating point. Both weighted arms are restricted to the eleven calibration-eligible classes, those with an eligible held-out exemplar. The certificate attaches to a decision instance, not to a class. Each arm is evaluated at 55 cells. Per-class weights, radii, residuals, and eligibility are reported in Table~\ref{tab:calibration_record}. The calibrated arms refuse the remainder rather than substitute untested weights.

\begin{table}[t]
\centering
\caption{Held-out calibration record: Chebyshev-centre weight $w^\dagger$
(box $[1,10^8]^3$), tier ratio $w_1{:}w_2{:}w_3$ (deployed heuristic:
$100{:}10{:}1$), inscribed radius $r^\dagger$, verified relative
stationarity residual, and eligibility per the pre-registered criteria
($r^\dagger>0$, genuine facet margins, residual $\le10^{-6}$). One
first-tick exemplar per class; two classes were not capturable on the
held-out logs (documented, with the protocol constants and the pinned
cache manifest, in the supplementary material).}
\label{tab:calibration_record}
\scriptsize
\setlength{\tabcolsep}{1.1pt}
\begin{tabular}{lrcrrl}
\toprule
Scenario class & $w^\dagger$ ($w_1$/$w_2$/$w_3$) & ratio & $r^\dagger$ & resid. & eligible \\
\midrule
changing\_lane & 9.99e7/6.65e4/6.65e4 & $1503{:}1{:}1$ & 6.65e4 & 1.9e-17 & yes \\
changing\_lane\_to\_left & 5e7/5e7/5e7 & $1{:}1{:}1$ & 5e7 & 6.1e-10 & yes \\
following\_lane\_with\_slow\_lead & 9.94e7/5.92e5/5.92e5 & $168{:}1{:}1$ & 5.92e5 & 5.8e-13 & yes \\
high\_magnitude\_speed & 5e7/5e7/5e7 & $1{:}1{:}1$ & 5e7 & 8.1e-9 & yes \\
medium\_magnitude\_speed & 316/156/156 & $2.0{:}1{:}1$ & 155 & 1.8e-7 & yes \\
near\_high\_speed\_vehicle & 5e7/5e7/5e7 & $1{:}1{:}1$ & 5e7 & 3.4e-8 & yes \\
near\_long\_vehicle & 316/156/156 & $2.0{:}1{:}1$ & 155 & 1.8e-7 & yes \\
near\_multiple\_vehicles & 5e7/5e7/5e7 & --- & 0 & 5.7e-2 & no$^{\mathrm{a}}$ \\
starting\_high\_speed\_turn & 779/370/370 & $2.1{:}1{:}1$ & 369 & 2.1e-7 & yes \\
starting\_left\_turn & 769/365/365 & $2.1{:}1{:}1$ & 364 & 2.2e-7 & yes \\
starting\_protected\_cross\_turn & 5e7/5e7/5e7 & --- & 0 & 1.2e-1 & no$^{\mathrm{a}}$ \\
starting\_right\_turn & 763/362/362 & $2.1{:}1{:}1$ & 361 & 2.2e-7 & yes \\
starting\_unprotected\_cross\_turn & 767/364/364 & $2.1{:}1{:}1$ & 363 & 2.2e-7 & yes \\
traversing\_intersection & 5e7/5e7/5e7 & --- & 0 & 7.4e-1 & no$^{\mathrm{a}}$ \\
\bottomrule
\multicolumn{6}{@{}l}{\rule{0pt}{8pt}$^{\mathrm{a}}$Infeasible centre; diagnosed in Section~\ref{sec:experiments:certificate}.}
\end{tabular}
\end{table}

All contrasts in this study are computed among its four conditions, which share solver settings, seeds, and scenario instances.

The calibration record, read for its tier ratios in Table~\ref{tab:calibration_record}, shows that the folklore separation $w_1\gg w_2\gg\cdots\gg w_L$ does not hold. Nine of eleven eligible classes have $1:1:1$ or approximately $2:1:1$ centers, compared to the deployed heuristic's $100:10:1$. The remaining two classes are more separated than the heuristic, with $168:1:1$ and $1503:1:1$. Near-equality is the typical geometry, but not a universal one. The ratios remain meaningful even where the region is broad. The box-sensitivity sweep of Section~\ref{sec:experiments:certificate} moves the normalized weight direction of every jointly eligible class by at most $0.5^\circ$ across four decades of box. The post-hoc audit of Section~\ref{sec:system:params} locates the heuristic's error, which runs opposite to folklore's intuition. In the two campaign classes whose broad regions actively constrain the weights, the region imposes a comfort-tier lower bound of approximately $89$, while $w_{\mathrm{dep}}$ supplies $10$. The heuristic gets the ordering correct and the magnitudes incorrect, in the direction that breaks equivalence. Priority is enforced by membership in $\Omega(p^\star)$, not by magnitude separation.

\subsection{Closed-Loop Dominance and Computational Cost}
\label{sec:calibrated:dominance}

The calibrated single solve is directional but not significant against either reference. For $C_5$ versus $C_H$, the result is 5 wins, 2 ties, and 4 losses. For $C_5$ versus $C_X$, the result is 8 wins, 0 ties, and 3 losses, with sign $p = 0.23$. The guarded condition lexicographically Pareto-dominates the cascade reference on nine of the eleven scenarios. For $C_6$ versus $C_X$, the result is 9 wins, 1 tie, and 1 loss. The exact two-sided sign test with the tie excluded yields $2 Pr[Bin(10, 1/2) >= 9] = 0.0215$. The pre-registered guarded hypothesis concerned the fallback rate, and this is one of three dominance comparisons reported here. The Bonferroni alpha is 0.05 divided by 3, approximately $0.017$. We therefore label this result exploratory: a strong hypothesis-generating observation, not a confirmatory finding.

The observation is mechanically coherent rather than an artifact of the dominance tolerances. The verdicts are unchanged when every per-level tolerance is scaled by 10 and sharpen when scaled by 0.1.
At least one $L_9$ win occurred at a per-scenario fallback rate of $0.00$, produced purely by accepted weighted solves. The effect is therefore not confined to episodes that effectively ran the cascade.
Per-scenario deciding levels, tolerance-scaling verdicts, and the full fallback-rate range are reported in the supplementary material. Per-seed verdicts agree with the seed-median verdict on 35 of 53 comparable seed pairs. This reflects seed variability that the exploratory label already acknowledges. A plausible mechanism, consistent with Section~\ref{sec:experiments:certificate}, is that exact stage-wise optimization at each tick is greedy over the episode and carries no closed-loop optimality guarantee.
A monitored surrogate that accepts certified weighted solves and falls back selectively can redistribute violation more favorably over the horizon. Establishing this causally is future work.

In cost, $C_5$ issues one weighted solve per tick, with median 1.14 s against the cascade's $10.85$ s. Both are unoptimized research prototypes, so the comparison is a relative scaling, not a real-time claim. $C_6$ invokes the hierarchy on $0.35$ of ticks, for a median 6.10 s per tick. The fallback rate exceeds the $0.20$ ceiling declared before the campaign. This is consistent qualitatively with the short certificate lifetime of Section~\ref{sec:experiments:certificate} and the monitor's imperfect operating point. Detailed fallback-rate accounting is provided in the supplementary material. The guard is a proof-of-concept selective-acceleration architecture, retaining about half of the single-solve timing advantage rather than all of it.

\subsection{Event-Level Behavioural Fidelity}
\label{sec:calibrated:events}

Calibration yields its strongest effect in the event record. Scoring all conditions on the same eleven classes and 55 paired cells with the study's cascade reference in Table~\ref{tab:calibrated_events}, calibration doubles $L_7$ event precision over the heuristic weight, from 0.232 to 0.463, lowers the $L_7$ false-negative rate from 0.72 to 0.56, and more than doubles $L_9$ precision, from 0.188 to 0.438. Both weighted arms produce fewer $L_9$ episodes than the cascade reference, 16 to 17 against 28. The calibrated gain at the collision level lies in the share of episodes matching the cascade's record. Event matching remains limited, with $R_{L_9} = 0.25$. Episode-count alignment should not be read as event-identity alignment. A scenario-cluster bootstrap with $10^4$ resamples over the eleven clusters supports the $L_7$ gain inferentially: $\Delta P_{L_7} = +0.231$, $95$ percent confidence interval $[+0.02, +0.48]$, and $\Delta R_{L_7} = +0.158$, confidence interval $[+0.02, +0.36]$, both excluding zero. The $L_9$ improvements are directional only, with Delta $P_{L_9} = +0.250$, confidence interval $[-0.22, +0.48]$. Collision-level events are too sparse on eleven scenarios for cluster-level inference.

Two limits bound the claim. Against the heuristic reference $C_H$, the calibrated advantage on the dominance count is not significant, with 5 wins, 2 ties, and 4 losses. On this benchmark, calibration is not established as an improvement over a well-tuned heuristic on that endpoint. The fallback rate exceeds the 0.20 ceiling declared before the campaign, consistent with the pointwise certificate lifetime measured in Section~\ref{sec:experiments:certificate}. The two campaigns' per-tick problems differ in the cascade velocity-cap construction. Contrasts are therefore computed within a campaign only, and the calibrated arms carry their own matched references.

\section{Validation and Heuristic Surrogate}
\label{sec:experiments}
The comparative campaign validates the apparatus: the reference cascade
realises the priority order it was written to enforce, and a heuristic
weight outside the equivalence region diverges from it exactly where the
certificate predicts.

\subsection{The cascade realises the order}
\label{sec:campaign:order}

Across the sixteen scenarios and five seeds, the cascade $C_4$
lex-Pareto-dominates the legacy baseline $C_0$ on $\mathbf{12}$ of
$\mathbf{16}$ scenarios; the exact one-sided binomial $p$-value at $k\ge12$
is $0.0384<\alpha=0.05$, with $95\%$ interval $[0.48,0.93]$ for the
benchmark win rate (Table~\ref{tab:deciding}; the dominance grid is in the
supplementary material). The research condition $C_1$
wins on $11$ of $16$ ($p=0.105$), a numerical majority one scenario short
of the threshold. The deciding level makes this a \emph{safety-level}
separation: seven of the twelve wins are decided at the vehicle-overlap
level $L_9$, three at $L_7$, and two at the headway level $L_3$. Resolved
per scenario (Fig.~\ref{fig:l9_scenarios}), the legacy baseline carries
non-zero seed-median $L_9$ overlap on seven of sixteen scenarios (medians
up to $9.09\,\mathrm{m^2\!\cdot\!s}$); the cascade drives six of the seven
to zero and sharply reduces the seventh ($1.245\!\to\!0.074$). Underneath, the
per-level decomposition (figure in the supplementary material) shows that,
under these benchmark conditions, the cascade achieved zero integrated
VRU-level ($L_{10}$) violation across all sixteen tested scenarios; $C_1$
did likewise. (This result is scoped to the specific reactive-traffic
benchmark conditions and does not constitute a claim of VRU collision
avoidance in real-world deployment.) The cascade keeps $L_9$ at or near
zero (a small nonzero seed-median on only two of sixteen scenarios), while
$C_1$ incurs nonzero $L_9$ on several scenarios. The cost surfaces at the
lowest-priority comfort level, where the median regression is $595.7\%$
--- the redistribution the lexicographic order prescribes.

The same structure appears without appealing to the pooled medians. The
per-scenario percent-reduction grid (Fig.~\ref{fig:snapshots_cross})
shows a left-to-right green-to-red gradient on a majority of scenario rows
--- the priority-respecting compromise visible per scenario rather than
only in aggregate --- and the level-major aggregate carries the same ordering,
$C_4\prec C_1\approx C_2\prec C_3\prec C_0$ at every level above $L_0$ and
reversing at $L_0$. At the vehicle-overlap level the aggregate median is
zero for $C_0$ and $C_4$, effectively zero ($2.6{\times}10^{-5}$) for
$C_2$, and only small-nonzero for $C_1$ ($0.0024$) and $C_3$ ($0.0590$).

The safety-level separation is also visible through the RSS proxy --- the
per-tick conjunction of the four collision-related observer rules. The
cascade attains the highest compliance rate, $0.8988$ against the
baseline's $0.8794$ ($+1.9$ percentage points), with the research
conditions statistically indistinguishable from the baseline. Per scenario
(Fig.~\ref{fig:rss_per_scenario}), $C_4$ attains the highest proxy
rate on $8$ of $16$ scenarios --- no research condition is highest on any
scenario in isolation --- and holds above $0.85$ on every scenario where
another condition falls below it. The residual non-compliant ticks
(minimum per-cell rates $0.195$--$0.322$) correspond to scenarios in which
nuPlan's reactive agents force the ego into a brief collision-rule
violation.

A crossed random-effects magnitude model
over all cells is consistent with the sign test --- it shows a statistically
distinguishable legal-level reduction and comfort-level increases, while the
safety-level effects point in the reducing direction with intervals
including zero; the framework-applicability
scans, the observer-only negative controls, and the smoothness ordering are
reported in the supplementary material. The independently authored certificate engine
corroborates the headline effect on its one calibrated rule pair --- vehicle
footprint overlap, $55/55$ events corroborated, zero disagreements --- and
independently reproduces the direction of the $L_9$ improvement. Throughout
this study, overlap is a geometric technical-evidence quantity computed over
bounding-box polygons under reactive simulated traffic; it is not a collision
determination, does not reflect fault attribution, is not indicative of
real-world collision frequency, and is scoped to this specific
\texttt{nuPlan-mini} benchmark rather than production deployment.

\begin{figure*}[t]
 \centering
 \includegraphics[width=0.9\textwidth]{
 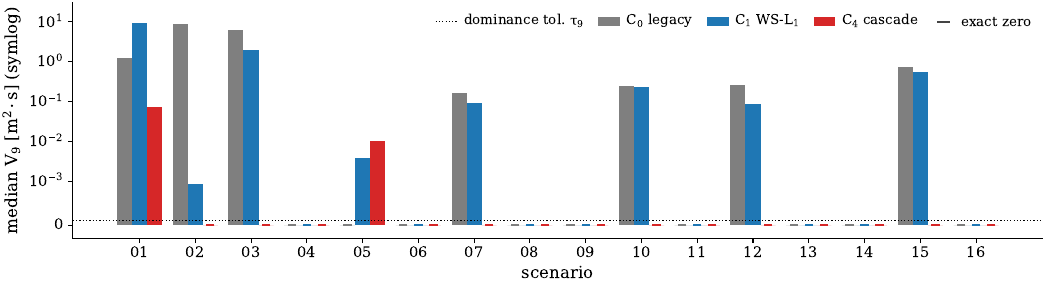}
 \caption{Deployment outcome at the vehicle-overlap level: per-scenario seed-median bounding-box overlap $V_9$ (symlog scale; baseline ticks denote exact zeros). The baseline $C_0$ (grey) incurs overlap on seven scenarios; the cascade $C_4$ (red) drives six to zero and sharply reduces the seventh (scenario $01$, $1.245\!\to\!0.074$), introducing only a small overlap on scenario $05$; the research condition $C_1$ (blue) tracks the baseline and exceeds it on scenario $01$ ($9.52$ vs $1.245$). The dotted line is the dominance tolerance $\tau_9$.}
 \label{fig:l9_scenarios}
\end{figure*}



\begin{table*}[t]
\centering

\begin{minipage}[t]{0.60\textwidth}
\vspace{0pt}
\centering

\includegraphics[width=\linewidth]{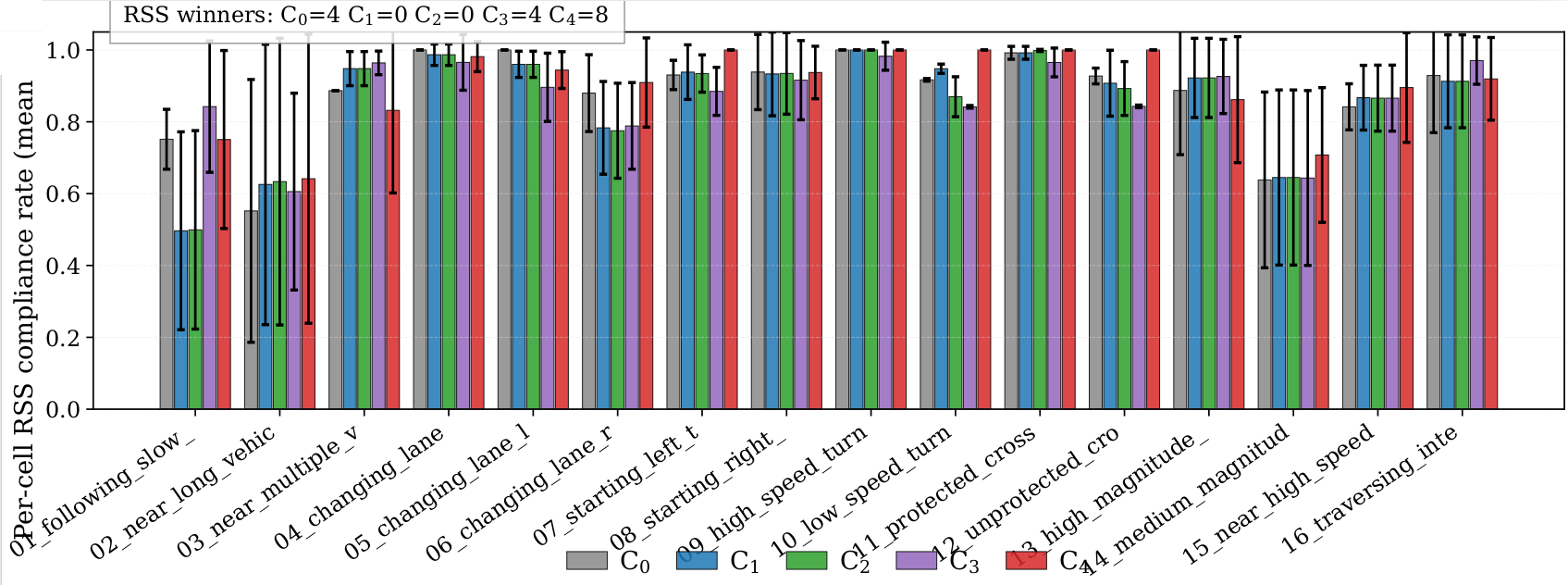}
\captionof{figure}{Deployment outcome on the RSS proxy: per-scenario compliance by condition (error bars: SD across $5$ seeds). $C_4$ wins on $8$ of $16$ scenarios, $C_3$ and $C_0$ on $4$ each, and $C_4$ holds above $0.85$ wherever another condition falls below it.}
 \label{fig:rss_per_scenario}

\end{minipage}
\hfill
\begin{minipage}[t]{0.35\textwidth}
\vspace{0pt}
\centering

\captionof{table}{Lex-Pareto dominance vs the legacy baseline with the exact one-sided binomial $p$ (wins of $n{=}16$), and the priority level at which each win is decided under the per-level tolerances. The cascade's majority is decided predominantly at the vehicle-overlap level $L_9$, not at the comfort levels.}
    \label{tab:deciding}
\footnotesize
\setlength{\tabcolsep}{1pt}
\begin{tabular}{lrrrrl}
\toprule
vs $C_0$ & W & T & L & $p$ & deciding level of wins (count) \\
\midrule
$C_1$ & 11 & 1 & 4 & 0.1051 & $L_{9}$:\,6, $L_{7}$:\,5 \\
$C_2$ & 5 & 2 & 9 & 0.9616 & $L_{9}$:\,3, $L_{7}$:\,2 \\
$C_3$ & 6 & 2 & 8 & 0.8949 & $L_{9}$:\,4, $L_{7}$:\,2 \\
$C_4$ & 12 & 0 & 4 & 0.0384 & $L_{9}$:\,7, $L_{7}$:\,3, $L_{3}$:\,2 \\
\bottomrule
\end{tabular}

\end{minipage}

\end{table*}

\subsection{The heuristic surrogate diverges where the certificate predicts}
\label{sec:campaign:surrogate}

The heuristic weight $w_{\mathrm{dep}}$ lies outside the computed
equivalence region in nearly every scenario class
(Section~\ref{sec:system:params}), and the certificate lifetime of
Section~\ref{sec:experiments:certificate} is one sampling interval --- so
the theory predicts tickwise proximity with event-level divergence, and
that is what the campaign measures. The five-seed $C_1$-vs-$C_4$ per-tick
compliance match has scenario medians of $1.000$ at $L_{10}$, $0.990$ at
$L_9$, and $1.000$ at the observer-only sentinels, yet fails the pre-stated
$0.95$ floor on subsets at $L_9$ and $L_8$; at event granularity the
surrogate produces more certificate overlap windows and more observer-side
$L_9$ episodes than the cascade. This event-level discrepancy and the
computational timing reported below are distinct empirical observations;
the study does not establish a causal relationship between them. Tickwise
and event-level findings differ because overlap events are sparse relative to episode length ---
sparse-event surrogate validation requires event precision and recall
alongside per-tick agreement.

The divergence has a stable per-scenario anatomy (single-condition
characterization figure in the supplementary material), and
Fig.~\ref{fig:snapshots_cross} grounds it in the scene: the per-rule
peak-violation record of the protected-cross episode under $C_1$, with
legal and sentinel rules firing around the crossing manoeuvre while
comfort rules absorb the sustained load. Under $C_1$, the top-violating
MPC-controlled rule is at $L_7$ in $11$ of $16$ scenarios (the
opposing-lane rule in $10$, the traffic-light rule in $1$), at $L_3$ in
$4$ (headway, speed limit, and lateral clearance), and at $L_0$ in $1$
(lateral comfort) --- matching scenario semantics: turn, intersection, and
lane-change scenarios are constrained by legal rules at $L_7$;
high-density following scenarios by headway and speed at $L_3$; one
sharp-turn scenario by lateral comfort at $L_0$. The per-rule,
per-scenario integrated violations exhibit a block structure --- turn and
intersection scenarios share a rule signature distinct from following
scenarios --- and this block structure is the empirical basis for keying
the calibration cache by scenario class. The stacked per-level breakdown
concentrates the violation mass at $L_7$ and $L_3$; the VRU-collision
level $L_{10}$ is zero in all scenarios, while the vehicle-overlap level
$L_9$ is nonzero in eight of sixteen, reaching
$9.52\,\mathrm{m^2\!\cdot\!s}$ in the slow-lead follow --- small relative
to the $L_7$/$L_3$ mass, but a genuine safety-level signal that precludes
characterizing the benchmark as collision-free.

\begin{figure*}[h!]
\centering
 \includegraphics[width=0.70\textwidth]{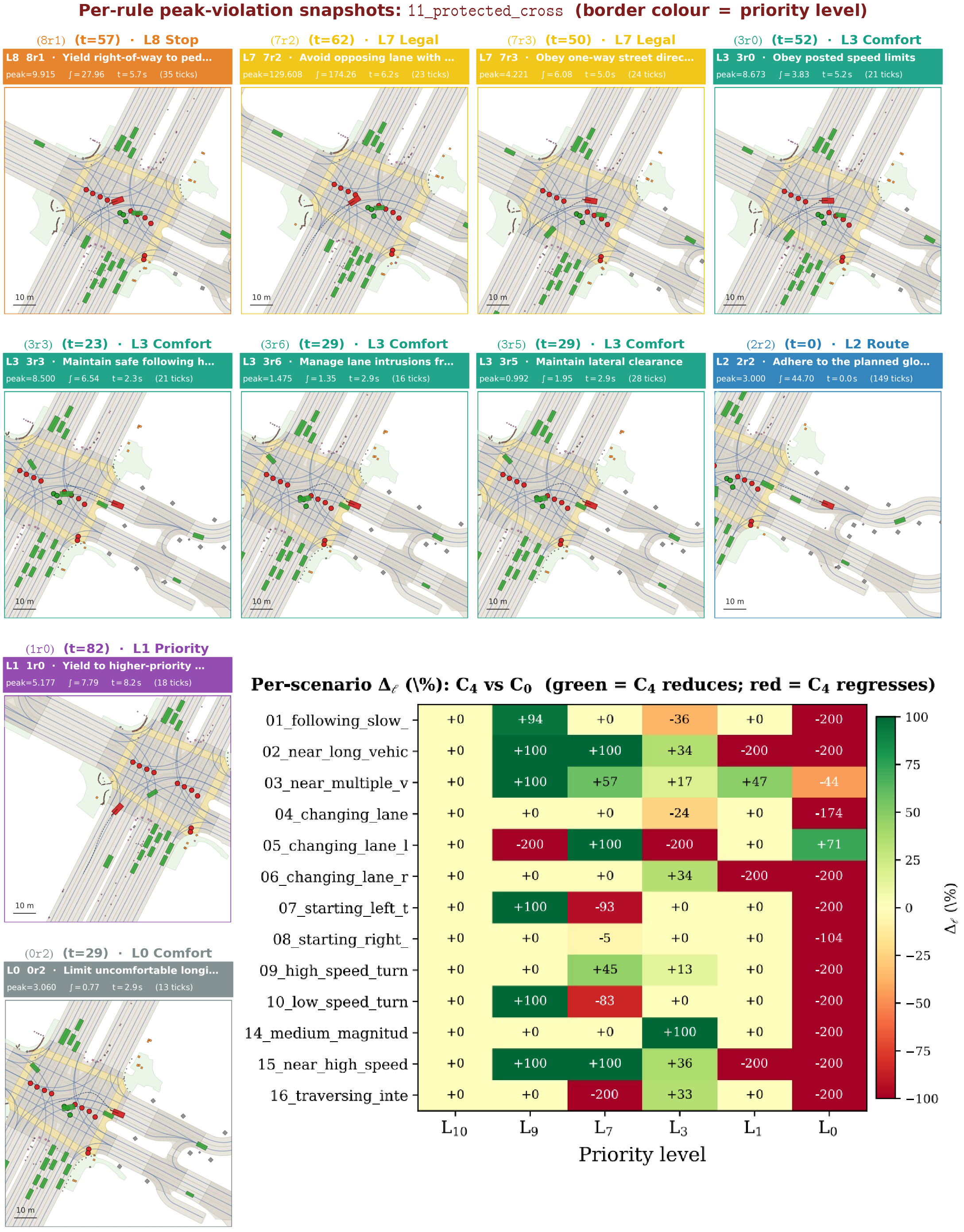}
 \caption{Closed-loop deployment in the scene: per-rule peak-violation
 snapshots for the protected-cross episode under the research condition
 $C_1$ (observer record; panel border color encodes the priority level).
 Each panel shows the intersection at one rule's peak-violation tick,
 annotated with the rule identifier, peak and integrated violation, peak
 time, and violated-tick count; the ego vehicle is red and surrounding agents
 are green. The lower-right panel reports the exact cascade's per-scenario,
 per-level percent reduction $\Delta_\ell$ of $C_4$ vs $C_0$ (rows: scenarios;
 columns: levels; green = reduction, red = regression). The left-to-right
 green-to-red gradient on a majority of rows is the per-scenario signature
 of priority-respecting compromise.}
 \label{fig:snapshots_cross}
\end{figure*}

In cost, the surrogate uses a median $16\%$ of the cascade wall time:
per-episode wall-time ratios span $1.76$--$9.39\times$ (median
$6.2\times$), with absolute per-tick times of the unoptimized research
prototypes at $0.68$--$2.92$\,s versus $5.08$--$10.27$\,s; per-scenario
times are tabulated in the supplementary material. Timing the cascade's
per-tick graph construction attributes a median $21\%$ of its cost to
construction, bounding the recoverable implementation speedup at
$1.27\times$; solver-convergence statistics are given in the supplementary
material.

\subsection{Closed-loop corroboration of the dichotomy}
\label{sec:campaign:dichotomy}

Both ablations contrast against the research condition $C_1$ on identical
scenario$\times$seed cells with paired statistics
(Table~\ref{tab:paired_contrasts}): per-level mean paired
differences with hierarchical-bootstrap $95\%$ intervals ($10^{4}$
replicates), an exact scenario-cluster sign-flip permutation test (all
$2^{16}$ assignments) with Westfall--Young max-$T$ adjustment across
levels, and the dominance predicate applied to the paired contrast. $H_3$
finds no effect: raising the SLP cap from one to three produces no
per-level difference distinguishable from zero after max-$T$ correction
(every adjusted $p\ge0.30$; paired dominance $4$W/$6$T/$6$L, sign
$p=0.754$), so the single-iteration setting loses nothing detectable.
$H_4$, by contrast, probes the practical consequence of the dichotomy at
deployment tolerances: switching the penalty to the squared hinge
significantly increases comfort-level violation ($L_0$:
$\bar{\Delta}=+0.58$, $95\%$ CI $[+0.20,+1.04]$, max-$T$ $p=0.003$; paired
dominance $3$W/$2$T/$11$L, sign $p=0.057$), while no safety-level
difference is statistically distinguishable from zero --- the practical
consequence Proposition~\ref{prop:l2tol} predicts, with the decay rates
themselves verified by the controlled sweep of
Fig.~\ref{fig:l1l2_decay}. A comparison against an external planner
(nuPlan's IDM), which argues for per-level evaluation rather than for the
equivalence theory, is reported in the supplementary material.

\begin{table*}[t]
\centering

\begin{minipage}[t]{0.59\textwidth}
\vspace{0pt}
\centering

\captionof{table}{Direct paired contrasts for the $H_3$/$H_4$ ablations
on identical scenario$\times$seed cells. Shown are the mean paired
difference in integrated violation (positive = ablation worse than
$C_1$), hierarchical-bootstrap 95\% CIs ($10^4$ replicates), exact
scenario-cluster sign-flip $p$-values, and Westfall--Young max-$T$
adjusted $p$-values. The bottom block reports paired seed-median
lex-dominance counts with exact two-sided sign tests.}
\label{tab:paired_contrasts}

\scriptsize
\setlength{\tabcolsep}{1pt}

\begin{tabular}{lrrrr@{\hspace{5pt}}rrrr}
\toprule
&
\multicolumn{4}{c}{$C_2-C_1$ (SLP-3, $H_3$)}
&
\multicolumn{4}{c}{$C_3-C_1$ ($\ell_2$, $H_4$)}
\\
\cmidrule(lr){2-5}
\cmidrule(lr){6-9}

Level
& $\bar{\Delta}$ & 95\% CI & $p$ & $p_{\max T}$
& $\bar{\Delta}$ & 95\% CI & $p$ & $p_{\max T}$ \\
\midrule

$L_{10}$
& +0.000 & [+0.000, +0.000] & 1.000 & 1.000
& +0.000 & [-0.015, +0.015] & 1.000 & 1.000 \\

$L_{9}$
& +0.136 & [-0.100, +0.529] & 0.744 & 1.000
& -0.089 & [-0.998, +0.735] & 0.602 & 1.000 \\

$L_{7}$
& +0.338 & [-1.211, +2.966] & 0.914 & 1.000
& -58.702 & [-131.587, +1.527] & 0.211 & 0.439 \\

$L_{3}$
& +0.082 & [-0.266, +0.523] & 0.637 & 1.000
& -6.884 & [-15.054, -1.619] & $<0.001$ & 0.154 \\

$L_{1}$
& +0.016 & [-0.002, +0.061] & 0.094 & 0.641
& +0.011 & [-0.027, +0.050] & 0.523 & 1.000 \\

$L_{0}$
& +0.056 & [-0.005, +0.162] & 0.103 & 0.307
& +0.577 & [+0.196, +1.035] & 0.002 & 0.003 \\

\midrule

Dominance vs.\ $C_1$
& \multicolumn{4}{c}{4W/6T/6L (sign $p=0.754$)}
& \multicolumn{4}{c}{3W/2T/11L (sign $p=0.057$)}
\\

\bottomrule
\end{tabular}

\end{minipage}
\hfill
\begin{minipage}[t]{0.39\textwidth}
\vspace{0pt}
\centering

\captionof{table}{Event-level behavioral fidelity to the cascade reference
on eleven calibration-eligible classes ($55$ paired cells per arm).
Episodes are debounced at $10$\,Hz and matched one-to-one at temporal
IoU $\geq0.5$; $P$ and $R$
measure fidelity to the cascade reference, not safety or quality.}
\label{tab:calibrated_events}

\footnotesize
\setlength{\tabcolsep}{1pt}

\begin{tabular}{llrrrr}
\toprule
Level & Condition & \#Ep. & $P$ & $R$ & FNR \\
\midrule

$L_9$
& $C_X$ cascade (ref.) & 28 & --- & --- & --- \\

&
$C_H$ heuristic & 16 & 0.188 & 0.107 & 0.893 \\

&
$C_5$ calibrated & 16 & \textbf{0.438} & \textbf{0.250} & \textbf{0.750} \\

&
$C_6$ guarded & 17 & 0.412 & 0.250 & 0.750 \\

\midrule

$L_7$
& $C_X$ cascade (ref.) & 57 & --- & --- & --- \\

&
$C_H$ heuristic & 69 & 0.232 & 0.281 & 0.719 \\

&
$C_5$ calibrated & 54 & \textbf{0.463} & \textbf{0.439} & \textbf{0.561} \\

&
$C_6$ guarded & 56 & 0.339 & 0.333 & 0.667 \\

\bottomrule
\end{tabular}

\end{minipage}

\end{table*}


\section{Discussion and Conclusion}
\label{sec:discussion}
The experiments expose two sources of weighted-sum/cascade divergence. First, different warm starts produce different locally convexified subproblems. Second, and more consequentially, equivalent-weight certification is decision-instance-specific: the sampled robust intersections are empty in all fourteen classes examined, native-rate lifetime is typically shorter than one tick, and persistence decreases
with active-set turnover. Consequently, a fixed weight may be exact at the calibration instance and invalid at the subsequent instance. Per-tick agreement saturates while the event record diverges; the compliance-pattern monitor provides an inexpensive but imperfect fallback signal. The calibrated study of Section~\ref{sec:experiments:calibrated} substantiates this: on the calibration-eligible classes, geometry-calibrated weights approximately double legal-level event precision relative to the heuristic weight at single-solve cost. The collision-level gain is directional with an interval containing zero. In an exploratory comparison, the guarded variant lexicographically Pareto-dominates the per-tick cascade. Per-tick exactness does not imply a closed-loop guarantee.

\subsection{Scope, Relation to RSS, and Future Work}
\label{sec:discussion:limits}
\phantomsection\label{sec:discussion:rss}%

The guarantees hold only within their stated scope: the global support-normal result requires convexity and existence; the explicit $\ell_1$ construction uses the scaled variables $\beta_{i,j}=w_i\alpha_{i,j}$, with augmented LICQ needed only for direct affine elimination and not for the general polyhedral projection; and the locality of the $\ell_2$ branch, together with the status of its first-order tolerance formula, is stated with Theorem~\ref{thm:l2asymp} and Proposition~\ref{prop:l2tol}.

Each sequentially linearised subproblem is a convex decision instance to which the certificate applies when those assumptions hold. Empirically, however, the study uses one vehicle model, one simulator, one instance of each of sixteen scenario types, and five reactive-traffic seeds, so the sign test is restricted to this fixed benchmark; the calibrated conditions cover eleven calibration-eligible classes rather than all
sixteen, and the advantage over a well-tuned heuristic weight is not established here. The reported $6.2\times$ speedup is a relative scaling between unoptimized research prototypes, both exceeding the nominal $100$ ms tick budget. An optimized exact-hierarchy baseline could recover the cascade's graph-construction share, approximately one fifth as detailed in Section~\ref{sec:campaign:surrogate}, and potentially more through solver reuse. Nevertheless, this is unlikely to close the gap.

The framework complements Responsibility-Sensitive Safety (RSS)~\cite{shalev2017rss,hasuo2022rss}: RSS pairs a Boolean safety condition with a proper response, whereas a rulebook orders violation objectives by minimum-violation arbitration~\cite{censi2019,tumova2013,veer2023,xiao2021}. The per-tick compliance
vector $b_{\boldsymbol{\epsilon}}(z)$ supplies tier-level Boolean summaries structurally analogous to an RSS condition check, though RSS's proper-response guarantee is the stronger object; the proxy of Section~\ref{sec:campaign:order} is a per-tick conjunction of collision-related observer rules, and a full RSS audit lies outside the present scope.

We identify five concrete directions for follow-on investigation that extend the present contribution. Each addresses a specific scope boundary of the current study and represents a natural next step in the research program rather than a deficiency in the claims made here: (i) stochastic priorities through chance constraints or an expected achievement map, whose upper image need not retain a tractable normal-region representation; (ii) inverse optimal control with candidate $\ell_1$ weights constrained to the equivalence region; (iii) adaptive recalibration triggered by compliance mismatches and active-set changes; (iv) comparisons with optimized hierarchical-QP and PDM-Closed~\cite{dauner2023pdm} baselines; and (v) a safety-stressed scenario subset constructed around severe, adversarial, or ego-initiated events with explicit fault attribution. The present results establish the cascade's $L_9$ separation only under the reactive \texttt{nuPlan-mini} benchmark; adversarial conditions are a natural follow-on.

\subsection{Conclusion}
\label{sec:conclusion}

Three words summarize what we have established about the weights at which one convex solver reproduces a lexicographic cascade. The certificate is \emph{computable}: a weighted sum has the lexicographic solution as a minimizer precisely when $(w,1)$ supports the upper image at $p^\star$, and under $\ell_1$ hinges the admissible slice is a polyhedron obtained by projecting a scaled-KKT system in
$\beta_{i,j}=w_i\alpha_{i,j}$, returning an interior weight and a certified margin whenever that slice has interior within the operator box. It is \emph{foreclosed} under smoothing: with $\ell_2$ squared hinges and a nonzero limiting multiplier, no finite weight is exact, and violation along the proved local-minimizer branch decays only as $O(1/w_i)$. And it is \emph{expiring}: in the evaluated receding-horizon instances the
robust intersection is empty in every class examined, and cross-tick validity tracks active-set stability.

The consequence is architectural rather than parametric. A certificate is issued for a decision instance, not for a controller. Consequently, persistence must be rechecked as the receding-horizon problem evolves. The deployable object is not a tuned weight vector, whether by separation folklore or otherwise, but a monitored single solve with selective cascade fallback. The compliance-pattern proxy triggers that fallback inexpensively and imperfectly; the measured sensitivity is $0.55$. Reliable and inexpensive expiry detection remains the open deployment problem.

The equivalence region reframes the practitioner's question. Rather than asking how much larger $w_1$ must be than $w_2$, the relevant inquiry is whether $w$ lies within the certified set at the current instance, and for what duration. The first question admits no answer; the second admits an algorithm, a margin, and a measured expiry.

\balance
\bibliographystyle{IEEEtran}
\bibliography{references}

\appendices
\section{Proofs of the $\ell_2$ Asymptotic Results}
\label{sec:appendix}
\label{app:l2}

This appendix supplies the local development behind
Theorem~\ref{thm:l2asymp} and Proposition~\ref{prop:l2tol}. Throughout,
$z^\star:=z_{\lex}^\star$; level~$i$ is fixed with $V_i(z^\star)=0$;
$\mathcal B:=\{j:g_{i,j}(z^\star)=0\}$ collects its boundary-binding
hinges; $\eta:=1/w_i$; and
$\Psi(z;w_{-i}):=J(z)+\sum_{i'\ne i}w_{i'}V_{i'}(z)$ collects the
remaining terms of the weighted objective.
The proof is stated for $z^\star\in\operatorname{int}\mathcal Z$. If a
fixed set of smooth hard constraints is active, the same argument applies
in local coordinates on their active manifold under joint LICQ and
second-order sufficiency; no claim is made across active-set changes.

\emph{(R1) Stable sign classes.}
On a neighbourhood of $z^\star$ intersected with the solution path
constructed below, every hinge outside $\mathcal B$ retains its sign class
while each hinge in $\mathcal B$ passes from binding to strictly violated,
so $\Psi$ is twice continuously differentiable there along the branch.

\emph{(R2) Linear independence.}
The gradients
$\{\nabla g_{i,j}(z^\star)\}_{j\in\mathcal B}$ are nonzero and linearly
independent; $G$ stacks them as rows.

\emph{(R3)--(R5) Regular limit problem.}
For every $w_{-i}$ in a compact set
$K\subset\mathbb R_{++}^{L-1}$, $z^\star$ is the locally unique solution of
$(\mathrm P_\infty)\colon\ \min_z\Psi(z;w_{-i})$ s.t.\
$g_{i,j}(z)\le0$, $j\in\mathcal B$,
with unique multipliers $\lambda^\star(w_{-i})>0$, second-order
sufficiency on the tangent space $\ker G$, and strict complementarity.

\begin{lemma}[No exact finite-weight stationary point]
\label{lem:app:noexact}
Under \textup{(R2)--(R5)}, if
$\lambda^\star(w_{-i})\ne0$, then $z^\star$ is not a stationary point of
$F_w(z):=\Psi(z;w_{-i})+w_iV_i(z)$
for any finite $w_i>0$. Since $z^\star\in\operatorname{int}\mathcal Z$, it
cannot be a local or global minimizer of the finite-weight smooth
penalized problem.
\end{lemma}

\begin{proof}
By Lemma~\ref{lem:vanish}, $\nabla V_i(z^\star)=0$, so
$\nabla F_w(z^\star)=\nabla\Psi(z^\star;w_{-i})
=-G^\top\lambda^\star(w_{-i})$,
the second equality by stationarity of $(\mathrm P_\infty)$. Because $G$
has full row rank and $\lambda^\star\ne0$, the right-hand side is nonzero,
so $z^\star$ fails the first-order necessary condition of the
finite-weight problem.
\end{proof}

\begin{theorem}[Local solution path and expansion]
\label{thm:app:path}
Under \textup{(R1)--(R5)}, there exist $\bar\eta>0$ and, for each
$w_{-i}\in K$, a continuously differentiable path
$\eta\mapsto\bigl(z(\eta),\lambda(\eta)\bigr)$, $\eta\in[0,\bar\eta)$,
with $z(0)=z^\star$ and
$\lambda(0)=\lambda^\star(w_{-i})$, such that for
$0<\eta<\bar\eta$, $z(\eta)$ is the locally unique stationary point and a
strict local minimizer of $F_w$ near $z^\star$ at $w_i=1/\eta$. Along this
branch,
\begin{equation}\label{eq:app:expansion}
\begin{aligned}
 g_{i,j}(z(\eta))
 &=
 \tfrac12\lambda_j^\star(w_{-i})\,\eta+O(\eta^2),\\
 V_i(z(\eta))
 &=
 \tfrac14\|\lambda^\star(w_{-i})\|_2^2\,\eta^2+O(\eta^3),
\end{aligned}
\end{equation}
with both remainders uniform over $w_{-i}\in K$.
\end{theorem}

\begin{proof}
On the branch specified in~(R1), every $j\in\mathcal B$ is strictly
violated and every other hinge retains its sign, so stationarity of $F_w$
reads
$\nabla\Psi(z;w_{-i})+\sum_{j\in\mathcal B}2w_ig_{i,j}(z)\nabla
g_{i,j}(z)=0$. Setting $\lambda_j:=2w_ig_{i,j}(z)$ and $\eta:=1/w_i$ turns
stationarity and definition into the square system
\begin{equation}\label{eq:app:phi}
 \mathcal H(z,\lambda;\eta,w_{-i})
 :=
 \begin{pmatrix}
  \nabla\Psi(z;w_{-i})+G(z)^\top\lambda\\[2pt]
  \bigl(g_{i,j}(z)-\tfrac{\eta}{2}\lambda_j\bigr)_{j\in\mathcal B}
 \end{pmatrix}
 =0,
\end{equation}
where $G(z)$ stacks the binding-hinge gradients; at $\eta=0$ this is the
active-set KKT system of $(\mathrm P_\infty)$ at
$(z^\star,\lambda^\star)$. Its Jacobian in $(z,\lambda)$,
$\bigl[\begin{smallmatrix}
 \nabla_{zz}^2\mathcal L(z^\star,\lambda^\star) & G^\top\\
 G & 0
\end{smallmatrix}\bigr]$,
is nonsingular because $G$ has full row rank and the Lagrangian Hessian is
positive definite on $\ker G$~\cite{nocedal2006}; the
implicit-function theorem gives the $C^1$ path. The second block
of~\eqref{eq:app:phi} holds identically along the path, so
$g_{i,j}(z(\eta))=\tfrac{\eta}{2}\lambda_j(\eta)
=\tfrac{\eta}{2}\lambda_j^\star+O(\eta^2)$; squaring and summing over
$\mathcal B$ gives the second expansion in~\eqref{eq:app:expansion}
(other level-$i$ hinges stay strictly satisfied), and strict
complementarity keeps the leading coefficients positive, validating the
assumed branch for sufficiently small $\eta>0$. The Hessian
of $F_w$ along the branch,
$\nabla^2\Psi+\sum_{j\in\mathcal B}\lambda_j\nabla^2g_{i,j}
+2w_i\sum_{j\in\mathcal B}\nabla g_{i,j}\nabla g_{i,j}^\top$, has its
first two terms converging to the Lagrangian Hessian of
$(\mathrm P_\infty)$, positive definite on $\ker G$, while the third is
positive and grows on the complementary normal subspace; hence $z(\eta)$
is a strict local minimizer for small $\eta>0$. Local uniqueness is the
implicit-function theorem again, and uniformity over $K$ follows from
compactness and continuous dependence of the KKT system and its inverse
Jacobian on $w_{-i}$.
\end{proof}



\end{document}